\documentclass{amsart}
\usepackage[utf8]{inputenc}
\usepackage[english]{babel}

\usepackage[margin=2.5cm]{geometry}
\usepackage{enumerate}
\usepackage{float}
\usepackage[final]{graphicx}
\usepackage{epstopdf} 

\usepackage{tabularx}
\newcolumntype{C}[1]{>{\centering\arraybackslash}m{#1}}

\usepackage{caption}
\usepackage{subcaption}
\usepackage{enumitem}
\usepackage{comment}
\usepackage{soul}

\usepackage{tikz-cd}
\usepackage{amsmath, multirow}
\usepackage{amsfonts}
\usepackage{amssymb}
\usepackage{amsthm}
\usepackage{hyperref}
\usepackage{cleveref}

\DeclareCaptionSubType[alph]{figure}
\usepackage[T1]{fontenc}
\usepackage{mathrsfs}
\usepackage{mathtools}

\usepackage{tikz}
\usetikzlibrary{decorations.markings}
\usetikzlibrary{decorations.pathreplacing}
\usetikzlibrary{arrows,shapes,positioning,patterns}
\usetikzlibrary{knots}
\tikzstyle{none}=[inner sep=0pt]
\pgfdeclarelayer{edgelayer}
\pgfdeclarelayer{nodelayer}
\pgfsetlayers{edgelayer,nodelayer,main}
\usepackage{url} 
\usepackage[all]{xy}

\newcommand{\op}{\mathrm{op}}

\newcommand{\bN}{\mathbb N}
\newcommand{\bK}{\mathbb K}

\newcommand{\bC}{\mathbb C}

\newcommand{\bZ}{\mathbb Z}

\renewcommand{\H}{\mathrm{H}}

\newcommand{\Mod}{\mathbf{Mod}}

\newcommand{\x}{\times}

\newcommand{\Cone}{{\mathfrak{Cone}~}}

\theoremstyle{plain}
\newtheorem{thm}{Theorem}[section]
\newtheorem{prop}[thm]{Proposition}
\newtheorem{lemma}[thm]{Lemma}
\newtheorem{cor}[thm]{Corollary}
\theoremstyle{remark}
\newtheorem{rem}[thm]{Remark}

\newtheorem{example}[thm]{Example} 
\newtheorem{q}[thm]{Question}

\theoremstyle{definition}
\newtheorem{defn}[thm]{Definition}

\newtheorem*{crit*}{Criterion~B}

\definecolor{aquamarine}{rgb}{0.5, 1.0, 0.83}
\definecolor{princetonorange}{rgb}{1.0, 0.56, 0.0}
\definecolor{caribbeangreen}{rgb}{0.0, 0.8, 0.6}
\definecolor{bunired}{rgb}{0.8, 0.0, 0.0}
\definecolor{cdgreen}{rgb}{0.0, 0.42, 0.24}
\definecolor{lavender(floral)}{rgb}{0.71, 0.49, 0.86}
\definecolor{bluedefrance}{rgb}{0.19, 0.55, 0.91}
\definecolor{iris}{rgb}{0.35, 0.31, 0.81}
\definecolor{darkgreen}{rgb}{0.33, 0.42, 0.18}

\newcommand{\tB}{{\tt B}}

\newcommand{\tG}{{\tt G}}
\newcommand{\tH}{{\tt H}}
\newcommand{\tK}{{\tt K}}
\newcommand{\tI}{{\tt I}}
\newcommand{\tD}{{\tt D}}
\newcommand{\tT}{{\tt T}}
\newcommand{\tR}{{\tt R}}
\newcommand{\tS}{{\tt S}}
\newcommand{\tP}{{\tt P}}
\newcommand{\tQ}{{\tt Q}}

\newcommand{\tA}{{\tt A}}
\newcommand{\tM}{{\tt M}}

\newcommand{\Comp}[2]{C_{#2}(#1)}

\DeclareMathOperator{\MH}{\mathrm{MH}}
\DeclareMathOperator{\MC}{\mathrm{MC}}
\DeclareMathOperator{\Hom}{\mathrm{Hom}}

\renewcommand{\bC}{\mathbf{C}}

\newcommand{\Graph}{\mathbf{Graph}}
\newcommand{\CGraph}{\mathbf{CGraph}}
\renewcommand{\Cone}{\text{Cone}}

\newcommand{\bD}{\mathbf{D}}
\newcommand{\cM}{\mathcal{M}}
\newcommand{\cP}{\mathcal{P}}
\newcommand{\cF}{\mathcal{F}}

\newcommand{\Digraph}{\mathbf{Digraph}}
\newcommand{\CDigraph}{\mathbf{CDigraph}}

\newcommand{\Rep}{\mathbf{Rep}}

\newcommand{\Quiver}{\mathbf{Quiver}}
\newcommand{\CQuiver}{\mathbf{CQuiver}}

\newcommand{\cV}{\mathcal{V}}

\title[Weak categorical quiver minor theorem and its applications]{The weak categorical quiver minor theorem and its applications: matchings, multipaths, and magnitude cohomology}

\author{Luigi Caputi}
\author{Carlo Collari}
\author{Eric Ramos}

\begin{document}

\maketitle

\begin{abstract}
Building upon previous works of Proudfoot and Ramos, and using the categorical framework of Sam and Snowden, we extend the weak categorical minor theorem from undirected graphs  to  quivers. As case of study, we investigate the consequences on the homology of  multipath complexes;
eg.~on its torsion. Further, we prove a comparison result: we show that, when restricted to directed graphs without oriented cycles, multipath complexes and matching complexes yield functors which commute up to a blow-up operation on directed graphs. We use this fact to compute the homotopy type of matching complexes for a certain class of bipartite graphs also known as half-graphs or ladders. We complement the work with a study of the (representation) category of cones, and with analysing related consequences on magnitude cohomology  of quivers.
\end{abstract}

\section{Introduction}

The graph minor theorem of Robertson and Seymour~\cite{zbMATH02134346}  states that the undirected graphs, partially ordered by the graph minor relationship, form a well-quasi-ordering.  If one restricts to graphs with bounded combinatorial genus, this fact has a categorical enhancement which was recently explored by Miyata and the third author in~\cite{miyata2023graph}. It is shown in~\cite{miyata2023graph} that the (weak) categorical minor theorem   is related to Noetherian properties of the (representation) category of undirected graphs with bounded genus. Remarkably, this relationship has non-trivial topological and combinatorial consequences. 
With a view to such consequences, in this paper we extend the (weak) categorical version of the graph minor theorem to quivers and directed graphs.
Inspired by previous works, and in particular by~\cite{trees, proudfoot2019functorial, proudfoot2022contraction}, we  borrow techniques from representation theory of categories. 
Sam and Snowden, in fact, developed a fascinating and powerful machinery to prove that a certain representation category is Noetherian~\cite{sam2017grobner}. Their approach makes use of combinatorial properties  --   called \emph{(quasi-) Gr\"obner} properties --  on the base category:  if a category $\bC$ is quasi-Gr\"obner, then its category of representations $\Rep_R\bC$, over a Noetherian ring $R$, is Noetherian.

The category $\Graph_{\leq g}^\op$ of undirected graphs with bounded genus $g$, and opposite minor morphisms, was shown to be quasi-Gr\"obner in \cite{proudfoot2022contraction}. 
As a consequence, all subrepresentations of finitely generated representations of $\Graph_{\leq g}^\op$ are also finitely generated.
This fact has some striking applications to the study of certain simplicial complexes associated to graphs, such as those arising from monotone properties of graphs, cf.~\cite{proudfoot2022contraction, miyata2023graph}.
The study of such simplicial complexes and, in particular, of matching complexes is a well-established~\cite{MR2022345, torsion,JONSSON20081504} and yet vibrant area of research~\cite{VZIPMatch,caterpillar,HighMatch,PolyMatch,DiscreteMatch}. In particular, the behaviour of the torsion appearing in the homology of matching complexes has been a subject of considerable attention -- see~\cite{torsion, MR2731551}, and references therein. 
Using the fact that the category $\Graph_{\leq g}^\op$ is quasi-Gr\"obner, Miyata and the third author approached the study of torsion of matching complexes from a different perspective. They proved, in \cite{miyata2023graph}, that if one considers graphs of bounded genus, then the homology of their matching complexes has universally bounded torsion. 
It is an open conjecture whether a similar statement holds true when considering the whole category of graphs, with no restriction on the genus -- see \cite[Conjecture~3.3]{miyata2023graph}.

In this work, we go beyond the undirected case, and  prove that also the category of quivers~$\Quiver$, and minor morphisms, is combinatorially well-behaved;

\begin{thm}[Weak categorical quiver minor theorem]\label{thm:contractcatfgintro}
    The category  $\Quiver_{\leq g}^\op$ is quasi-Gr\"obner. 
\end{thm}

As a consequence, simplicial complexes associated to quivers of bounded genus, such as those arising from monotone properties of directed graphs, are finitely generated (under some additional mild assumptions). 
As case of study we consider \emph{multipath complexes}~\cite{jason}, which also include cycle-free chessboard complexes~\cite{Omega}.
The interest in studying the homology of multipath complexes stems from their subtle relations with Hochschild homology~\cite{turner} and symmetric homology of algebras~\cite{AultFed,AultNoFed}.
Furthermore, multipath complexes are related to, and sometimes coincide with, matching complexes, see~\cite{monotonecohm22}.
In the spirit of \cite{miyata2023graph}, we use Theorem~\ref{thm:contractcatfgintro}  to investigate  the global behaviour of the homology of multipath complexes. Let $X\colon \Quiver^\op\to \textbf{SimplComp}$ be the functor associating to each quiver its multipath complex. 
First, we prove that, for any fixed $i,g\in \bN$, the homology $\mathrm{H}_i(X(-);\bZ)\colon \Quiver^\op_{\leq g} \to \Mod_R$ is a finitely generated  $\Quiver^\op_{\leq g}$-module (Proposition~\ref{prop:homologyfgmulti}). Then, we show that the torsion in $\mathrm{H}_i(X(-);\bZ)$ is universally bounded for quivers of bounded genus (Proposition~\ref{prop:torsionmultipaths}). 

In a different direction, we analyse further the connection between matching and multipath complexes.
To this end, we introduce the notion of blow-up~$B(\tG)$ of a directed graph~$\tG$, see~Definition~\ref{def:blowup}.
We show, in Theorem~\ref{thm:multi=match}, that the multipath complex of a directed graph $\tG$ without oriented cycles and the matching complex of the underlying undirected graph of $B(\tG)$ are isomorphic. 
The blow-up construction defines a functor $\mathbf{B}$ from the category of directed graphs without oriented cycles (and injective morphisms of digraphs) $\Digraph_o$ to the opposite of the category of undirected graphs (and minor morphisms)~$\Graph^{\op}$. Then, Theorem~\ref{thm:multi=match} can be restated by saying that the diagram 
    \begin{equation}\label{eq:commdgm}
	\begin{tikzcd}
		\Digraph_o\arrow[r,"\mathbf{B}"]\arrow[dr,"X"']  & \Graph^{\op}\arrow[d,"M"]\\
		&  \mathbf{SimpCompl}
	\end{tikzcd}     
    \end{equation}
is commutative up to isomorphism of simplicial complexes; see also Corollary~\ref{cor:commdiagram}. Here, we have denoted with~$M$ the functor associating to each undirected graph its matching complex. We infer the following direct comparison result:

\begin{thm}\label{thm:agreementintro}
Multipath complexes of directed graphs without oriented cycles are isomorphic to (joins of) matching complexes of bipartite graphs.
\end{thm}

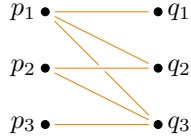
\begin{figure}
\begin{tikzpicture}[scale = 1.5]
    \node (a) at (0,1) {};
    \node (b) at (0,0.5) {};
    \node (c) at (0,0) {};
    \node (ap) at (1,1) {};
    \node (bp) at (1,0.5) {};
    \node (cp) at (1,0) {};
    
    \node at (0,1) [left] {$ p_1$};
    \node at (0,0.5) [left] {$p_2$};
    \node at (0,0) [left] {$p_3$};
    \node at (1,1) [right] {$q_1$};
    \node at (1,0.5) [right] {$q_2$};
    \node at (1,0) [right] {$q_3$};

    \draw[black, fill] (a) circle (.035);
    \draw[black, fill] (b) circle (.035);
    \draw[black, fill] (c) circle (.035);
    \draw[black, fill] (ap) circle (.035);
    \draw[black, fill] (bp) circle (.035);
    \draw[black, fill] (cp) circle (.035);

    \draw[orange] (a) -- (ap);
    \draw[orange] (a) -- (bp);
    \draw[orange] (a) -- (cp);
        \draw[white, fill] (0.5,0.5) circle (.035);
    \draw[orange] (b) -- (bp);
    \draw[orange] (b) -- (cp);
    \draw[orange] (c) -- (cp);    
\end{tikzpicture}
\caption{The bipartite graph $\tB_3$.}
\label{fig:blowupintro}
\end{figure}
  
Diagram~\eqref{eq:commdgm} and Theorem~\ref{thm:agreementintro} show that, when restricting to digraphs without oriented cycles, computations of matching complexes can be obtained by equivalent computations of multipath complexes, in turn providing a way to compute the homotopy type of new families of matching complexes. We give here a main example. Let $\tB_n$ be the undirected graph --  {sometimes called half-graph} or ladder~\cite{NESETRIL2021103223} -- on $2n$ vertices $p_1, \dots, p_n$ and $q_1, \dots, q_n$, and edges $(p_i, q_j)$ for all~$i\leq j$ -- see Figure~\ref{fig:blowupintro}.

\begin{thm}\label{thm:matchingof tr}
    The matching complex of $\tB_n$ is either contractible or homotopy equivalent to a wedge of spheres. 
\end{thm}

We complement the paper with two more main directions. The first, regarding  the theoretical description of the combinatorial properties of the category of quivers, aims to shed light on the analog of  \cite[Conjecture~3.3]{miyata2023graph} for quivers. Trying to enlarge the category for which the weak categorical minor theorem is true, we prove in Proposition~\ref{prop:cones} that the category of cones of graphs with bounded genus is also quasi-Gr\"obner.  This is beneficial to (partially) get around the restriction of bounded genus. For instance, the cone category includes such families as the wheel graphs, which are cones over cycles, and the Thagomizer graphs, which are cones over star trees. The second direction concerns some applications of the categorical machinery of Sam and Snowden to homology theories of digraphs which do not (directly) arise from monotone properties of digraphs. Our main example is \emph{magnitude homology} of graphs~\cite{richardHHA}. Magnitude homology is the categorification of magnitude, as defined by Leinster -- see, for example, the book~\cite{LeinsterBook} -- capturing  many interesting properties of graphs~\cite{gomi2019magnitude, MR4250513, asao, asao_hiraoka_kanazawa_2023, tajima2023causal}. In \cite{torsionCC}, the first two authors proved that magnitude cohomology of undirected graphs is finitely generated; in turn, getting new insight about the torsion in magnitude homology -- see also \cite{MR4275098,SadzanovicSummers}. We extend in Section~\ref{sec:magnitude} the results of \cite{torsionCC} to the category of quivers. Using the work of Asao~\cite{asao}, we also get in Corollary~\ref{cor:torionPH} some insight on the torsion of the path homology of quivers, introduced in~\cite{Grigoryan}.

We conclude with a perspective on homology theories of digraphs and quivers, whose interest in recent years has skyrocketed due to connections to persistence and topological data analysis. 
If \cite[Conjecture~3.3]{miyata2023graph} were true, we could extend the results of this paper to all (directed) graphs, without restriction on the genus. In this work, we focused on two main examples of cohomology theories of simplicial complexes deriving from quivers. However, in view of a better understanding of the (weak) categorical minor theorem, it would be interesting to investigate  also further cohomology theories of graphs, quivers and categories. To name a few, some recent intriguing examples, related to magnitude homology, are for example reachability homology~\cite{hepworth2023reachability, caputi2023reachability}, spectral homology~\cite{ivanov2023nested}, or those introduced in~\cite{ivanov2022simplicial,roff2023iterated}.

\subsection*{Conventions and notation}

All (directed) graphs and quivers are assumed to be finite, and are denoted in typewriter font, e.g.~$\tQ,\tG$.  Unless otherwise specified, $R$ will denote a Noetherian commutative ring with identity.
All categories are in bold font. We mainly use the following categories of graphs:~\\

\begin{tabular}{ll}
$\Graph$ & undirected graphs and minor morphisms\\
$\Digraph$ & digraphs and minor morphisms\\
$\Quiver$ & quivers and minor morphisms\\
$\CGraph$ & undirected graphs and contractions\\
$\CDigraph$ & digraphs and contractions\\
$\CQuiver$ & quivers and contractions\\
$\Digraph_o$ &  directed graphs without directed cycles and injective morphisms of digraphs \\
$\Graph_g$ & undirected graphs of genus $g$ and contractions\\
$\Digraph_g$ & digraphs of genus $g$ and contractions\\
$\Quiver_g$ & quivers of genus $g$ and contractions\\
$\CGraph_{\leq g}$ & undirected graphs of genus at most $g$ and contractions\\
$\CDigraph_{\leq g}$ & digraphs of genus at most $g$ and contractions\\
$\CQuiver_{\leq g}$ & quivers of genus at most $g$ and contractions\\
$\Graph_{\leq g}$ & undirected graphs of genus at most $g$ and minor morphisms\\
$\Digraph_{\leq g}$ & digraphs of genus at most $g$ and  minor morphisms\\
$\Quiver_{\leq g}$ & quivers of genus at most $g$ and  minor morphisms\\
$\Cone(\Quiver_{\leq g})$ & cone over quivers of genus at most $g$ and minor morphisms of the base quiver
\end{tabular}~\\~\\
While we listed them here for the readers' convenience, the precise definition of these categories will be given throughout the paper.

\section{Basic notions}

In this section, we recall some basic notions and set the notations needed throughout.

\subsection{Quivers}
Recall that a (finite) \emph{quiver} $\tQ$ consists of two finite sets $V(\tQ)$ and $E(\tQ)$, together with a pair of functions
$s,t\colon E(\tQ) \to V(\tQ)$. 
The elements of~$V(\tQ)$ are called \emph{vertices}, while the elements of~$E(\tQ)$ are called~\emph{edges}.
For each $e\in E(\tQ)$ the vertices~$s(e) $ and~$t(e)$ are the \emph{source} and the \emph{target} of $e$, respectively, and they are collectively called the \emph{endpoints} of~$e$.  
A \emph{morphism of quivers} $f \colon \tQ \to {\tt Q'}$ is given by a pair of functions $f_V\colon V(\tQ) \to V(\tQ')$ and $f_E\colon E(\tQ) \to E(\tQ')$, such that the two  diagrams 
	\begin{center}
	\begin{tikzcd}
		E(\tQ)\arrow[d,"f_E"'] \arrow[r, "s"] & V(\tQ) \arrow[d,"f_V"]\\
		E(\tQ') \arrow[r, "s'"] & V(\tQ')
	\end{tikzcd}
 \quad
 and \quad
	\begin{tikzcd}
		E(\tQ)\arrow[d,"f_E"'] \arrow[r, "t"] & V(\tQ) \arrow[d,"f_V"]\\
		E(\tQ') \arrow[r, "t'"] & V(\tQ')
	\end{tikzcd}  
	\end{center}	
commute. Note that morphisms of quivers can send an edge to a self-loop.

Special cases of quivers are \emph{digraphs}. These are quivers such that each edge $e$ is completely determined by the (ordered) pair of its endpoints $(s(e), t(e))$. An \emph{undirected graph} $\tG$ is a quiver together with a map $r\colon E(\tG)\to E(\tG)$, called \emph{reflection}, which is an involution such that $s\circ r =t$.
An {\em undirected edge} in an undirected graph refers to a pair of edges in the underlying quiver which are swapped by the reflection. 
Note that any morphism of quivers between two undirected graphs commutes with the respective reflections.

\begin{rem}
We will often identify each edge $e$ in a digraph with the ordered pair $(s(e), t(e))$.
Similarly, if the quiver underlying an undirected graph is a digraph, we will identify each undirected edge with the unordered collection of its endpoints. 
\end{rem}

A \emph{subquiver}~$\tH$ of a quiver~$\tQ$ is a quiver such that $V(\tH)\subseteq V(\tQ)$, $E(\tH)\subseteq E(\tQ)$, and both source and  target functions are given by restriction. If $\tH$ is a subquiver of~$\tQ$, we  write $\tH \leq \tQ$. If~$\tH$ is a subquiver of $\tQ$, and $V(\tH) = V(\tQ)$, then we  say that $\tH$ is  \emph{spanning} in~$\tQ$.
A subquiver of a digraph is automatically a digraph. However, a subquiver of an undirected graph may not be an undirected graph. Subquivers of undirected graphs which are undirected graphs with reflection defined by restriction are called \emph{subgraphs}. 
To be coherent with the classical terminology, we will refer also to subquivers of digraphs as \emph{subgraphs}.

A quiver $\tQ$ has a geometric realization~$|\tQ|$. This is the geometric realization of the CW complex whose $0$-cells are the vertices of the quiver, $1$-cells are the edges of the quiver, and the attaching maps are given by the source and target maps. 
In the case of undirected graphs we will take the quotient by the action of the reflection. 
We define the \emph{genus} of a quiver (resp.~undirected graph) as the first Betti number of its geometric realisation.
A (\emph{directed}) \emph{tree}~$\tT$ is an undirected graph (resp.~quiver) whose genus is $0$ or, equivalently, whose geometric realization $|\tT|$ is contractible.

Given a quiver $\tQ$ and an edge $e\in E(\tQ)$, we can define two quivers~$\tQ\setminus e$ and $\tQ/e$, called respectively the \emph{deletion} and the  \emph{contraction} of $e$.
The former is the subquiver of $\tQ$ obtained by removing $e$ from the set of edges. 
The latter is the quiver whose edges are $E(\tQ)\setminus e$, whose vertices are the quotient of $V(\tQ)$ by the identification $t(e) = s(e)$, and whose source and target maps are defined as follows: $s(e') = [s(e')]$ and $t(e') = [t(e')]$, where the brackets denote the equivalence class.
Less formally,~$\tQ/e$ is obtained from (the geometric realization of) $\tQ$ by contracting $e$ to a point.
Contraction and deletion of undirected graphs are defined similarly; the only change is that we remove both $e$ and $r(e)$, and then we define the reflection as the map induced by the original reflection.

Note that the operation of contracting edges does not change the homotopy type of the geometric realization, unless the edge contracted is a self loop.
We are not going to consider this latter case, and only allow contractions of edges with distinct endpoints. Similarly, when dealing with deletions, we will only allow deletions  whose geometric realization is connected.

A minor of a quiver (resp.~undirected graph) $\tQ'$ is a quiver (resp.~undirected graph)~$\tQ$ that is isomorphic to a quiver (resp.~undirected graph) obtained from~$\tQ'$  by iterative contractions and deletions. More formally, we have the following definition of minor morphisms of quivers.

\begin{defn}
A \emph{minor morphism} $\phi\colon \tG'\to\tG$ of quivers is a map of sets
\[
\phi\colon V(\tG')\sqcup E(\tG')\sqcup \{\star\}\to V(\tG)\sqcup E(\tG)\sqcup \{\star\} \ ,
\]
such that:
\begin{itemize}
\item $\phi(V(\tG'))=V(\tG)$ and $\phi(\star)=\star$;
\item if an edge $e \in E(\tG')$ has endpoints $(s(e),t(e)) = (v, w)$,  
and $\phi(e)\neq \star$,
then either $\phi(e) = \phi(v) = \phi(w)$ is a vertex of $\tG$,
or $\phi(e)$ is an edge of $\tG$ with endpoints $s(\phi(e)) = \phi(v)$ and $ t(\phi(e)) = \phi(w)$;
\item there is a bijection between $\phi^{-1}(E(\tG))$ and $E(\tG)$;
\item for each vertex $v\in \tG$, the preimage $\phi^{-1}(v)$ as a subquiver of $\tG'$ is a directed tree.
    \end{itemize}
If there is a minor morphism $\phi\colon \tG'\to\tG$ we will say that $\tG$ is a \emph{minor} of $\tG'$. The definition of {\em minor morphism of undirected graphs} is the same, but the words ``edge'', ``directed tree'', and ``subquiver'' are replaced by ``undirected edge''`, ``tree'', and ``subgraph'', respectively.
\end{defn}

The preimage of $\star$ under $\phi$ consists of deleted edges, whereas the edges that are mapped to vertices of $\tG$ represent the contracted ones. 
Furthermore, the last item in the definition implies that self loops cannot be contracted, but only deleted.

A  \emph{simple path}, or \emph{directed path}, in a quiver~$\tQ$  is subquiver whose edges can be ordered $e_1,...,e_n$ in such a way that (i)~$s(e_{i+1})=t(e_i)$ for each $i<n$, (ii) no vertex is encountered twice, i.e.~if $s(e_i) = s(e_j)$ or $t(e_i) = t(e_j)$ then $i=j$, and (iii)~$s(e_1)\neq t(e_n)$.
A subquiver of $\tQ$ satisfying properties (i) and (ii), but not (iii), in the definition of simple path is called \emph{oriented cycle}.
We call \emph{alternating} any quiver $\tQ$ such that the sets $t(E(\tQ)), s(E(\tQ))\subset V(\tQ)$ are disjoint.

\begin{example}
The alternating quiver $\tA_n$ in Figure~\ref{fig:alternating} has a simple path $\tI_{m}$, with $m \leq \lfloor \frac{n}{2}\rfloor$ edges, as a minor (but not as a subquiver).  More generally, the simple path $\tI_1$ is a minor of any directed tree. 
\end{example}

We get the categories $\Graph$ and $\Quiver$ of undirected graphs and quivers, respectively, with minor morphisms. The category  $\Digraph$ is the full subcategory of $\Quiver$ spanned by digraphs. Analogously, we denote by $\CGraph$ and $\CQuiver$ the categories of graphs and quivers, respectively, with only contractions as allowed morphisms. The category $\CDigraph$ is also defined, as the full subcategory  of $\CQuiver$ spanned by digraphs.

\begin{rem}\label{rem:adjunctiongraphs}
There is a pair of functors $\iota\colon \Quiver \to \Graph$ and  $\rho\colon \Graph\to \Quiver$.
The latter is just the forgetful functor which forgets the reflection. Intuitively, $\rho$ makes each undirected edge of a graph bidirectional.
The functor~$\iota$ instead replaces each edge with an undirected edge. We  will refer to $\iota(\tQ)$ as the underlying graph of $\tQ$. 
\end{rem}

  \begin{figure}[h]
    \centering
    \begin{tikzpicture}[baseline=(current bounding box.center)]
		\tikzstyle{point}=[circle,thick,draw=black,fill=black,inner sep=0pt,minimum width=2pt,minimum height=2pt]
		\tikzstyle{arc}=[shorten >= 8pt,shorten <= 8pt,->, thick]
		
		\node[above] (v0) at (0,0) {$v_0$};
		\draw[fill] (0,0)  circle (.05);
		\node[above] (v1) at (1.5,0) {$v_1$};
		\draw[fill] (1.5,0)  circle (.05);
		\node[above] (v2) at (3,0) {$v_2$};
		\draw[fill] (3,0)  circle (.05);
		\node[above] (v4) at (4.5,0) {$v_3$};
		\draw[fill] (4.5,0)  circle (.05);
		\node[above] (v5) at (6,0) {$v_4$};
		\draw[fill] (6,0)  circle (.05);
		\node[above] (v6) at (7.5,0) {$v_5$};
		\draw[fill] (7.5,0)  circle (.05);
		 \node[above]  at (9,0) {$v_{n-1}$};
		 \node (v7) at (9,0) {};
		\draw[fill] (9,0)  circle (.05);
		
		\node   at (8.25,0) {$\dots$};
		 \node (v8) at (10.5,0) {};
		 \node[above]  at (10.5,0) {$v_n$};
		\draw[fill] (10.5,0)  circle (.05);
		
		\draw[thick, bunired, -latex] (0.15,0) -- (1.35,0);
		\draw[thick, bunired, -latex] (2.75,0) -- (1.65,0);
		\draw[thick, bunired, latex-] (4.35,0) -- (3.15,0);
		\draw[thick, bunired, latex-] (4.65,0) -- (5.85,0);
		\draw[thick, bunired, latex-] (7.35,0) -- (6.15,0);
	    \draw[thick, bunired, ] (v7) -- (v8);

	\end{tikzpicture}
	\caption{The alternating linear quiver $\tA_n$ on $n+1$ vertices. The edge between $v_{n -1} $ and $v_{n}$ can be oriented either way depending on the parity of $n$. }
    \label{fig:alternating}
\end{figure}
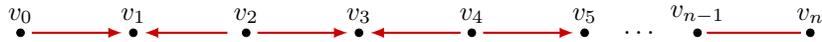

\subsection{Finitely generated $\bC$-modules}

Let $\bC$ be a (essentially) small 
category, and $A$ be a ring. 
A \emph{representation} of~$\bC$, or \emph{${\bf C}$-module}, \emph{over~$A$} is a functor $\mathcal{M}\colon\mathbf{C} \to A\text{\bf -Mod}$ with values in the category of (left) $A$-modules. Denote by ${\Rep}_A(\mathbf{C})$ the category of $\bC$-modules over~$A$ where morphisms are given by natural transformations.

A \emph{submodule} of a $\bC$-module $\mathcal{M}$ is a $\bC$-module $\mathcal{N}$ such that $\mathcal{N}(c)$ is a $A$-submodule of~$\mathcal{M}(c)$, for each object~$c\in\bC$.
If $S$ is a subset of $\bigoplus_{c\in \bC}\cM(c)$, the \emph{span of~$S$}, denoted by~$\mathrm{span}(S)$, is the minimal $\bC$-submodule of~$\cM$ containing~$S$. 

\begin{defn}\label{def:fgmod}
A $\mathbf{C}$-module $\mathcal{M}$ is \emph{finitely generated} if there is a finite set  $S \subseteq \bigoplus_{c\in \bC} \mathcal{M}(c)$, such that $\mathrm{span}(S)=\cM$. 
\end{defn}

We can also characterise finitely generated modules in terms of simpler modules.
For each object $c$ of $\bC$, define a \emph{principal projective $\bC$-module}~$\cP_c$, as follows
\[
\cP_c(c')\coloneqq A\langle\Hom_{\bC}(c,c')\rangle,\quad \text{for each }c'\in\bC \ ,
\]
{i.e.}~the free (left) $A$-module spanned by~$\Hom_{\bC}(c,c')$, and $\cP_c$ is then defined on morphisms by (post)composition. Given a morphism $\gamma\colon c \to c'$, we denote by $e_\gamma$ the corresponding element in $\cP_c(c')$.

\begin{lemma}[{\cite[Proposition~2.3]{FI_Noetherian}}]\label{lem:fingengen}
A $\bC$-module $\cM$ is finitely generated if and only if there exists a surjection
    \[
\bigoplus_{i=1}^n \cP_{c_i} \to \cM
    \]
for some objects $c_1,\dots,c_n$ of $\bC$.
\end{lemma}

For a finitely generated $\bC$-module $\cM$, we refer to the objects  $c_1,\dots,c_n$ of $\bC$ in the lemma as \emph{generators} of $\cM$.

\begin{defn}
A $\mathbf{C}$-module $\mathcal{M}$ is \emph{Noetherian} if all its submodules are finitely generated. The category~${\Rep}_A(\mathbf{C})$ is {(locally) Noetherian} if all finitely generated $\mathbf{C}$-modules over $A$ are Noetherian. 
\end{defn}

Observe that, in discussing properties related to finite generation, it is often possible to restrict to principal projective modules. Indeed, by \cite[Proposition~3.1.1]{sam2017grobner}, the category~${\Rep}_A(\mathbf{C})$ is Noetherian if and only if every principal projective module $\cP_c$ is Noetherian.

\begin{example}\label{ex:FI}
Let~$\mathbf{FI}$ be the category of finite sets and injective maps.  Then, by~\cite[Theorem A]{FI_Noetherian}, the category~${\Rep}_R(\mathbf{FI})$ is Noetherian, for any (commutative Noetherian) ring $R$. 
\end{example}

Noetherian properties of various other categories have been extensively investigated, in particular thanks to the techniques developed by Sam and Snowden in \cite{sam2017grobner}. One of the main results in the latter paper is that, for a given category~$\bC$ (under some combinatorial assumptions), and a (possibly non-commutative) left Noetherian ring $R$, the associated category of representations is Noetherian as well. Before recalling  the combinatorial conditions to be required on the category~$\bC$, and stating the main result in this section, we need a definition. Let $\cF\colon \bC\to \bD$ be a functor {between essentially small categories}.

\begin{defn}\label{def:propertyF}
 We say that $\cF$ satisfies \emph{property (F)} if for every object $d\in \bD$ there exist finitely many objects $c_1,\dots,c_n$ of $\bC$, and morphisms $\delta_i\colon d\to \cF(c_i)$, such that: for any object~$c$ in~$\bC$,  and morphism $\delta\colon d\to \cF(c)$, there exists a morphism $\gamma_i\colon c_i\to c$ satisfying $\delta=\cF(\gamma_i)\circ \delta_i$.
\end{defn}

Observe that a functor that is surjective on both objects and morphisms satisfies property~(F).
Property (F) allows us to transfer finitely generated properties through functors. In fact, the following  holds: 

\begin{prop}[{\cite[Proposition~3.2.3]{sam2017grobner}}]   If a functor $\cF\colon \bC\to \bD$ satisfies property~(F), and $\cM\colon \bD\to R\text{\bf -Mod}$ is finitely generated, then
$\cF^*\cM\coloneqq \cM\circ\cF\colon \bC\to R\text{\bf -Mod}$ is finitely generated.  
\end{prop}

We need to introduce some further notation and terminology;  see also~\cite[Section~4.1]{sam2017grobner} for a more extensive overview.
Let  ${\bf S}\colon \bC\to \mathbf{Set}$  be a functor with values in the category of sets. We associate a poset $|{\bf S}|$ to ${\bf S}$ as follows. First, take~$\widetilde{{\bf S}}$ to be the union  
$\bigcup_{c\in \bC} {\bf S}(c)$. Then, for an element~$f$ in~${\bf S}(c)$ and an element~$g$ in~${\bf S}(c')$, set $f\leq g$ if and only if there exists a morphism $h\colon c\to c'$ in $\bC$ such that  ${{\bf S}(h)(f)=}h_* (f)= g$. Consider the equivalence relation~$\sim $ defined by~$f\sim g$ if and only if $f\leq g$ and $g\leq f$. Then, the poset~$|{\bf S}|$ is defined as the quotient of the set~$\widetilde{{\bf S}}$ with respect to~$\sim$, equipped with the partial order induced  from~$\leq$.

\begin{defn}\label{def:orderable}
An \emph{ordering} on ${\bf S}\colon \bC\to \mathbf{Set}$ is a choice of a well-order on~${\bf S}(c)$ for each $c$ in $\bC$, such that for every morphism $c\to c'$ the induced map ${\bf S}(c)\to {\bf S}(c')$ is strictly order-preserving; in such a case, we say that ${\bf S}$ is \emph{orderable}. 
\end{defn}

We say that  a poset~$P$ is Noetherian if, for every infinite sequence $x_1, x_2, \dots$ in $P$, there exist indices $i<j$ such that $x_i\leq x_j$ ({cf.}~\cite[Proposition~2.1]{sam2017grobner}). 
We can now recall the fundamental definition of Gr\"obner category.

\begin{defn}\label{def:Grobner}
An essentially small category~$\bC$ is \emph{Gr\"obner} if, for all objects $c$ of $\bC$, the functor ${\bf S}_c\coloneqq \Hom_{\bC}(c, -)$ is orderable, and the associated poset $|{\bf S}_c| $ is Noetherian.
An essentially small category~$\bC$ is \emph{quasi-Gr\"obner} if there exists a Gr\"obner category $\widetilde \bC$ and an essentially surjective functor $\widetilde \bC\to\bC$ satisfying  property~(F).
\end{defn}

Roughly speaking, a category is said to be \emph{Gr\"obner} if its slice categories are restrictive enough, so that representations of the category allow for a theory of ``Gr\"obner bases''. 
The category ${\bf FI}$ from Example~\ref{ex:FI} is not a Gr\"obner category, as its automorphism groups are
symmetric groups, hence non-trivial. However, ${\bf FI}$ is quasi-Gr\"obner:

\begin{rem}\label{rem:OI}
    Let $\mathbf{OI}$ be the category of linearly ordered finite sets and  ordered inclusions. This is a ``rigidified'' version of the category ${\bf FI}$; in particular, it has trivial automorphism groups. It was shown in \cite[Theorem~7.1.2]{sam2017grobner} that $\mathbf{OI}$ is indeed a Gr\"obner category, and that the functor $\mathbf{OI}\to {\bf FI}$ is an essentially surjective functor satisfying property (F). As a consequence, ${\bf FI}$ is quasi-Gr\"obner. 
\end{rem}

The next result follows readily from the definitions:

\begin{prop}\label{prop:qGrobner}
Let $\cF\colon \bC\to\bD$ be a functor satisfying  property (F). If $\bC$ is a quasi-Gr\"obner category, and $\cF$ is essentially surjective, then $\bD$ is a quasi-Gr\"obner category.
\end{prop}

\begin{proof}
Composition of essentially surjective functors is essentially surjective. Moreover, the composition of functors satisfying property (F) satisfies property (F) by \cite[Proposition~3.2.6]{sam2017grobner}. 
\end{proof}

Assume now that $R$ is a left Noetherian ring. The following is one of the main results connecting combinatorial properties of a category with the Noetherianity of the category of representations.

\begin{thm}[{\cite[Theorem~1.1.3]{sam2017grobner}}]\label{rem:qGrobner is Noeth}
If $\bC$ is a quasi-Gr\"obner category, then the category~${\Rep}_A(\mathbf{C})$ is Noetherian.
\end{thm}

In particular, if $R$ is a left Noetherian ring and $\bC$ is  quasi-Gr\"obner, all submodules and quotients of finitely generated functors $\mathcal{M}\colon\mathbf{C} \to R\text{\bf -Mod}$ are also finitely generated.

\section{Gr\"obner properties of $\Quiver$ and $\Digraph$}\label{sec:Grobner_digraph}

The primary goal of this section is to prove that certain subcategories of $\Quiver^{\op}$ and $\Digraph^{\op}$
are Noetherian.

\begin{defn}
    Let $g$ be a non-negative integer. The category $\Quiver_g$ (resp. $\Quiver_{\leq g}$) is the full subcategory of $\Quiver$ whose objects are quivers of genus $g$ (resp.~at most $g$). The categories $\Digraph_g$ and $\Digraph_{\leq g}$ are defined in the same fashion.
\end{defn}

Observe that in both $\Quiver_g$ and $\Digraph_g$, the only minor morphisms that appear are automorphisms and contractions, as deletions never preserve genus. On the other hand, among the morphisms in $\Quiver_{\leq g}$ and $\Digraph_{\leq g}$ there are both deletions and contractions.

For technical reasons, we need to define ``rigidified'' versions of the above categories. These will play a prominent role in the proof of this section's main theorem.
First, recall that a \emph{rooted spanning tree} for a quiver $\tG$ is a pair $(\tT,v_r)$, where $\tT$ is a spanning subquiver of $\tG$ which is a directed tree, and $v_r$ is a fixed vertex of $\tT$ called {\em root}. 
Note that, just as any connected subquiver, $\tT$ is obtained from $\tG$ by subsequent deletions.
We set $P\Quiver_g$ to be the category whose objects are \emph{undirected} graphs of genus $g$, enhanced with the following extra information:
\begin{enumerate}
    \item a choice of rooted spanning tree;
    \item a choice of planar embedding for the aforementioned spanning tree;
    \item a choice of orientation for the edges outside of this spanning tree;
    \item at each non-root vertex, a label indicating whether that vertex is the source or target of the (unique) edge of the spanning tree leading from the vertex to the root.
\end{enumerate}

The morphisms of this category will be edge contractions (and automorphisms) of the underlying graphs, which preserve all of the above structure. We similarly define $P\Digraph_{g}$.

\begin{rem}
    The reader might find it strange that we defined the category $P\Quiver_g$ by starting with undirected graphs, and then adding extra information that effectively orients them. The reason for this relates with our ultimate proof that this category is Gr\"obner. The idea is that we start with planar trees, and then add the data of the edge orientations and the extra edges in such a way that they can be encoded into the vertices of the tree. This way, we will be able to use a version of \textit{Kruskal's Tree theorem}, which allows for the vertices of the tree to be labelled.
\end{rem}

\begin{rem}
    We observe that our category $P\Quiver_g$ is a minor modification of the category $\mathcal{P}\mathcal{G}_g$ in \cite{proudfoot2022contraction}. Indeed, the main difference between the two categories lies in the labels on the vertices of the spanning tree. The important point is that this extra data is \emph{finite} in nature; this means that our vertices are being given labels from a finite set (source and target). This will allow us to apply \cite[Corollary 3.7]{proudfoot2022contraction}.\\
\end{rem}

To prove that representations of $\Quiver_{\leq g}^{\op}$ and $\Digraph_{\leq g}^{\op}$ have our desired Noetherian property, we will use the Gr\"obner methods of Sam and Snowden \cite{sam2017grobner}. To summarize, for each fixed $g$, we will proceed as follows:

\begin{itemize}
    \item We show that the category $P\Quiver_g^{\op}$ is Gr\"obner;
    \item we show that the forgetful functor $P\Quiver_g^{\op} \rightarrow \Quiver_g^{\op}$ has Property (F) and is essentially surjective;
    \item finally, we argue that the natural inclusion $\bigsqcup_{q \leq g}\Quiver_q^{\op} \hookrightarrow \Quiver_{\leq g}^{\op}$ has Property (F) and is essentially surjective.
\end{itemize}

Once we have completed each of the above steps, \cite[Theorem 1.1.3]{sam2017grobner} (cf.~Theorem~\ref{rem:qGrobner is Noeth}) will immediately imply our desired Noetherian Property.

Note that in our outline we have not said anything about the category $\Digraph_{\leq g}$! To justify why, we have the following standard proposition.

\begin{prop}
    The category $P\Digraph_{\leq g}^{\op}$ is Gr\"obner provided that $P\Quiver_{\leq g}^{\op}$ is.
\end{prop}

\begin{proof}
By construction the category $P\Digraph_{\leq g}$ is a full subcategory of $P\Quiver_{\leq g}$. The Gr\"obner property is inherited by full subcategories per \cite[Proposition 4.4.2]{sam2017grobner}.
\end{proof}

The above proposition justifies why the first step in our outline need not be repeated for the two categories. We will see that the latter two steps can be proven identically in the digraph case, and therefore it is unnecessary to repeat them. In particular, the above represents the crux of why we will only be working with the quiver categories from this point forward in this section.

\begin{prop}\label{prop:GrobnerQuiver}
    For any $g \geq 0$ the category $P\Quiver_g^{\op}$ is Gr\"obner.
\end{prop}

\begin{proof}
In \cite[Theorem 3.10]{proudfoot2022contraction} it is proven that the category $\mathcal{P}\mathcal{G}_{g,S}^{\op}$ is Gr\"obner where $S$ is any finite (or even well-quasi-ordered) set of permissible vertex labels. Consider the case wherein $S = \{s,t\}$. $P\Quiver_g$ is a full subcategory of $\mathcal{P}\mathcal{G}_{g,S}^{\op}$ in this case. It follows that $\mathcal{P}\mathcal{G}_{g,S}^{\op}$ is Gr\"obner by the same argument given in the prior proposition.
\end{proof}

Following our outline, our next goal is to now show that the forgetful functor $P\Quiver_g \rightarrow \Quiver_g$ has Property~(F) and is essentially surjective. It is clear that every quiver in $\Quiver_g$ is in the image of the forgetful functor. Property (F) follows immediately from \cite[Lemma 3.11]{proudfoot2022contraction}. To conclude our proof it therefore remains to show that the embedding $\bigsqcup_{q \leq g}\Quiver_q^{\op} \hookrightarrow \Quiver_{\leq g}^{\op}$ is essentially surjective with property (F).

\begin{lemma}\label{lem:leqg}
Fix $g \geq 0$. The embedding  $\bigsqcup_{q \leq g}\Quiver_q^{\op} \hookrightarrow \Quiver_{\leq g}^{\op}$ is essentially surjective with property (F).
\end{lemma}

\begin{proof}
Let $\tQ$ be a quiver of genus at most $g$, and let $\{\phi_i\colon \tQ_i \rightarrow \tQ\}$ be an exhaustive list of all possible \emph{deletions} that terminate at $\tQ$. Note that by a deletion we mean a minor morphism for which no edge is contracted and at least one edge is deleted. Further note that because our genus $g$ is fixed, and because deletions always strictly decrease genus when applied, this collection is finite.

To prove that our embedding has (F), we must argue that for any quiver $\tQ'$ of genus at most $g$ and any minor morphism $\psi\colon\tQ' \rightarrow \tQ$ we can find an $i$ and an edge contraction $\eta\colon\tQ' \rightarrow \tQ_i$, such that $\psi = \phi_i \circ \eta$. Indeed, this just amounts to the statement that any minor morphism can be written as a composition of a deletion followed by an edge contraction.
\end{proof}

Let going forward, we will write $\CQuiver_{\leq g}$ to denote the category $\bigsqcup_{q \leq g}\Quiver_q^{\op}$, of quivers of genus at most~$g$ and contractions. To summarize the work of this section, we have proven the following.

\begin{thm}\label{thm:contractcatfg}
    The categories  $\Quiver_{\leq g}^\op$, $\CQuiver_{\leq g}^\op$, and $\Digraph_{\leq g}^{\op}$ are quasi-Gr\"obner. 
\end{thm}

    Consider the \emph{edge module} 
    \[
\mathcal{E}_{g}\colon (\Quiver_{\leq g})^\op\to \Mod_R 
    \]
    which associates to a quiver $\tQ$ the free $R$-module $\mathcal{E}_{g}(\tQ)$ generated by the set~$E(\tQ)$  of edges of $\tQ$. For any minor morphism, there is a well-defined inclusion of edges in the opposite direction of the contraction. This in turn induces maps between the aforementioned free modules.

\begin{prop}\label{prop:edgemodfg}
Let $d$ be a natural number.    Then, the module $\mathcal{E}_{g}^{\otimes d}$  is a finitely generated $\Quiver_{\leq g}^\op$-module.
\end{prop}

\begin{proof}
The proof follows the same arguments of \cite[Lemma~4.2]{miyata2023graph}; by replacing the sets of edges with the sets of \emph{directed} edges. 
\end{proof} 

Although defined in a similar way, note that the vertex module
\[
\cV\colon \Quiver_{\leq g}^\op \to \Mod_R
\]
is not finitely generated. However, we have the following:

\begin{prop}\label{thm:Vkfg}
{Let $R$ be a ring with identity.} Then, ${\rm Hom}(\mathcal{V}^{\oplus k},R)$ is finitely generated, as $\Quiver^{\op}_{\leq g}$-module, for each $k$. 
\end{prop}

\begin{proof}
The proof follows the same arguments of \cite[Theorem~3.11]{torsionCC}, with only two small changes: the edges of $\tR(m)$ have to be oriented, and the quiver $\tS(m_1,m_2,m_2', m_3)$ in Figure~\ref{fig:gens-vertex-comodule}  replaces the graph $\tS(m_1,m_2,m_3)$ used in~\cite{torsionCC}.
\begin{figure}
    \centering
    \begin{tikzpicture}[thick, scale = 1.3]
    \draw[fill] (0,0) circle (.05);
    \draw[fill] (2,0) circle (.05);
    \node (a) at (0,0) {};
    \node (b) at (2,0) {};

    \draw[red, -latex] (a) .. controls +(.5,.5) and +(-.5,.5) .. (b);
    \draw[red, -latex] (a) .. controls +(.1,1) and +(-.15,0) .. (1,1) .. controls +(.15,0) and +(-.1,1) .. (b);
    \draw[blue, -latex] (a) -- (b);
    \draw[red, -latex] (b) .. controls +(-.5,-.5) and +(.5,-.5) .. (a);
    \draw[red, -latex] (b) .. controls +(-.1,-1) and +(-.15,0) .. (1,-1) .. controls +(.15,0) and +(.1,-1) .. (a);
    \draw[red, -latex] (a) .. controls +(-.5,-.5) and +(-.5,.5) .. (a);
    \draw[red, -latex] (a) .. controls +(-.25,-1) and +(0,-.5) .. (-1,0); \draw[red ] (-1,0) .. controls +(0,.5) and +(-.25,1) .. (a);
    \draw[red, -latex] (b) .. controls +(.5,-.5) and +(.5,.5) .. (b);
    \draw[red, -latex] (b) .. controls +(.25,-1) and +(0,-.5) .. (3,0); 
    \draw[red ] (3,0) .. controls +(0,.5) and +(.25,1) .. (b);

    \node[below]  at (2.7,0) {$\dots$};
    \node[below]  at (-.7,0) {$\dots$};
    \node  at (-.7,0.2) {$m_1$};
    \node  at (2.7,0.2) {$m_3$};
    
    \draw[line width = 3, color = white] (-1.1,0) .. controls +(0,.15) and +(0,-.15) .. (-.7,0.1) .. controls +(0,-.15) and +(0,.15) .. (-.3,0);
    \draw (-1.1,0) .. controls +(0,.15) and +(0,-.15) .. (-.7,0.1) .. controls +(0,-.15) and +(0,.15) .. (-.3,0);

    \draw[line width = 3, color = white] (3.1,0) .. controls +(0,.15) and +(0,-.15) .. (2.7,0.1) .. controls +(0,-.15) and +(0,.15) .. (2.3,0);
    \draw (3.1,0) .. controls +(0,.15) and +(0,-.15) .. (2.7,0.1) .. controls +(0,-.15) and +(0,.15) .. (2.3,0);

    \node[left]  at (1,.775) {$\vdots$};
    \node[left]  at (1,-.625) {$\vdots$};
    \node[right]  at (1,.7) {$m_2$};
    \node[right]  at (1,-.7) {$m'_2$};
    
    \draw[line width = 3, color = white] (.95,1.1) .. controls +(.15,0) and +(-.15,0) .. (1.05,.7) .. controls +(-.15,0) and +(.15,0) .. (.95,.3);
    \draw (.95,1.1) .. controls +(.15,0) and +(-.15,0) .. (1.05,.7) .. controls +(-.15,0) and +(.15,0) .. (.95,.3);
    
    \draw[line width = 3, color = white] (.95,-1.1) .. controls +(.15,0) and +(-.15,0) .. (1.05,-.7) .. controls +(-.15,0) and +(.15,0) .. (.95,-.3);
    \draw (.95,-1.1) .. controls +(.15,0) and +(-.15,0) .. (1.05,-.7) .. controls +(-.15,0) and +(.15,0) .. (.95,-.3);
    \end{tikzpicture}
    \caption{The graph $\tS(m_1,m_2,m_2', m_3)$. The edge in blue goes from $v_1$ to $v_2$.}
    \label{fig:gens-vertex-comodule}
\end{figure}
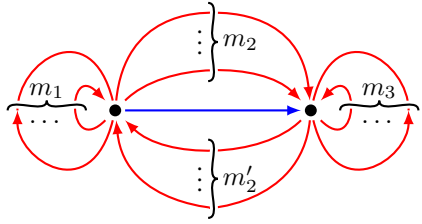
\end{proof}

\subsection{Categories of directed cones}

For a given quiver $\tG$, a \emph{cone} over $\tG$, denoted $\Cone(\tG)$, is a quiver obtained by adding a new vertex to $\tG$, while also adding (directed) edges to connect it with every vertex of $\tG$. Note that for our purposes, the added \emph{cone point} is connected to each vertex of $\tG$ by exactly one edge. In other words, we do not allow for multi-edges connected to the cone point.

Let $\tG,\tG'$ be two quivers. Then a \emph{base contraction} from $\Cone(\tG)$ to $\Cone(\tG')$ is a minor morphism $\Cone(\tG) \rightarrow \Cone(\tG')$ involving the contraction of edges of $\tG$, thought of as living inside of $\Cone(\tG)$, followed by the deletion of the newly created multiedges that are connected to the cone point of $\Cone(\tG')$. Similarly, a \emph{base deletion} is a minor morphism $\Cone(\tG) \rightarrow \Cone(\tG')$ involving the deletion of edges of $\tG$, thought of as a subquiver of $\Cone(\tG)$. More generally, a \emph{base minor morphism} is any minor morphism $\Cone(\tG) \rightarrow \Cone(\tG')$ that can be written as a composition of base contractions and base deletions. We will write $\Cone(\Quiver_{\leq g})$ for the category of directed cones whose base graph has genus at most $g$, and whose morphisms are base minor morphisms.

Our first result will tell us that representations of the cone category satisfy a Noetherian Property.

\begin{prop}\label{prop:cones}
    The category $\Cone(\Quiver_{\leq g})^{\op}$ is quasi-Gr\"obner. In particular, subrepresentations of finitely generated representations of $\Cone(\Quiver_{\leq g})^{\op}$ are also finitely generated. 
\end{prop}

\begin{proof}
    We begin with the category $P\Quiver_{\leq g}^{\op}$, and modify its objects to add an extra label to each vertex. Specifically, this label will be either 'u' (for up) or 'd' (for down). We also modify our morphisms to insist that the newly added label must be preserved. This new category will be Gr\"obner following an essentially identical proof to Proposition \ref{prop:GrobnerQuiver}.

    Consider now the functor from our new category to $\Cone(\Quiver_{\leq g})^{\op}$, which sends every object to the cone over its underlying graph, while orienting the cone edges in the direction dictated by the vertex labels. A morphism in our category naturally defines a morphism in $\Cone(\Quiver_{\leq g})^{\op}$ by performing the analogous base deletions and contractions. In the case of a contraction, the cone edges which are deleted by the corresponding base contraction are precisely those whose vertices were originally further from the selected spanning tree's root.

    Our functor is essentially surjective by constuction, and it is immediate that every base deletion and contraction appears in the image of some map. In particular, our functor has property (F), as desired.
\end{proof}

One benefit of the cone category is that it allows us to (partially) get around the restriction of bounded genus. For instance, $\Cone(\Quiver_{\leq g})^{\op}$ includes such families as the Wheel graphs, which are cones over cycles, and the Thagomizer graphs, which are cones over star trees.

\begin{rem}
    The idea to expand categories of bounded genus by considering cones was first done by Proudfoot and the third author in \cite{proudfoot2019functorial}. In that work, only undirected trees are considered. This work will therefore expand previous applications both by allowing for higher genera bases, and also allowing for the edges - cone or otherwise - to be directed.
\end{rem}

For our applications in the sections that follow, we will need to state an analogous result to Proposition~\ref{prop:edgemodfg}. Once again the proof is similar to that of \cite{miyata2023graph}.

\begin{prop}\label{prop:ConeEdge}
The $\Cone(\Quiver_{\leq g})^{\op}$-module,
\[
\tG \mapsto \bZ^{\vert E(\tG)\vert}
\]
is finitely generated. Moreover, the same is true of its tensor powers.
\end{prop}

\section{Applications to multipath complexes}\label{sec:multipath}
In this section, we apply the abstract technology of Section~\ref{sec:Grobner_digraph} to simplicial complexes arising from monotone properties of directed graphs~\cite{Jonsson}. More specifically, we will deal with multipath complexes~\cite{jason}.

\subsection{Multipath complexes and their torsion}

We are interested in the monotone property of digraphs on disjoint sets of simple paths; following~\cite{turner, primo}, we call them multipaths:

\begin{defn}\label{def:multipaths}
A \emph{multipath} of a quiver~$\tG$ is a spanning subquiver such that each connected component is either a vertex or a simple path. The \emph{length} of a multipath is the number of its edges. 
\end{defn}

The \emph{path poset} of a quiver $\tG$ is the poset $(P(\tG),\leq )$, where $P(\tG)$ is the set of multipaths in $\tG$ (including the multipath with no edges), ordered by the relation of being a subgraph. 
To the path poset $P(\tG)$ we associate a simplicial complex: 

\begin{defn}[{\cite[Definition~6.2]{secondo}}]\label{def:pathcomplx}
For a quiver~$\tG$, the \emph{multipath complex} $X(\tG)$ is the simplicial complex whose face poset (augmented to include the empty simplex~$\emptyset$) is the path poset $P(\tG)$. 
\end{defn}

\begin{figure}[h]
\centering
	\begin{tikzpicture}[baseline=(current bounding box.center),line join = round, line cap = round]
		\tikzstyle{point}=[circle,thick,draw=black,fill=black,inner sep=0pt,minimum width=2pt,minimum height=2pt]
		\tikzstyle{arc}=[shorten >= 8pt,shorten <= 8pt,->, thick]
		\def\c{8}\def\d{.5}
		\node[above] (v0) at (0-\c,0+\d) {$v_0$};\draw[fill] (0-\c,0+\d)  circle (.05);
		\node[above] (v1) at (1.5-\c,0+\d) {$v_1$};\draw[fill] (1.5-\c,0+\d)  circle (.05);
		\node[above] (v2) at (3-\c,0+\d) {$v_{2}$};\draw[fill] (3-\c,0+\d)  circle (.05);
		\node[above] (v3) at (4.5-\c,0+\d) {$v_{3}$};\draw[fill] (4.5-\c,0+\d)  circle (.05);
		
		\draw[thick, bunired, -latex] (0.15-\c,0+\d) -- (1.35-\c,0+\d);
		\draw[thick, bunired, -latex] (1.65-\c,0+\d) -- (2.85-\c,0+\d);
		\draw[thick, bunired, -latex] (3.15-\c,0+\d) -- (4.35-\c,0+\d);
		
		\node (e1) at (0,0) {};
		\node (e2) at (2,0) {};
		\node (e3) at (60:2) {};
		\draw[pattern=north west lines, pattern color=bluedefrance,thick] (e1.center) -- (e2.center) -- (e3.center) -- (e1.center);
		\node[circle,fill=bunired,scale=0.5] at (e1) {};\node[below] at (e1) {$(v_0,v_1)$};
        \node[circle,fill=bunired,scale=0.5] at (e2) {};\node[below] at (e2) {$(v_1,v_2)$};
        \node[circle,fill=bunired,scale=0.5] at (e3) {};\node[above] at (e3) {$(v_2,v_3)$};

    \end{tikzpicture}
	\begin{tikzpicture}
        \def\y{1}\def\x{3}
        \tikzstyle{sp}=[circle,fill=black,scale=0.3]
        \tikzstyle{se}=[thick, bunired, -latex]
        \node (p123) at (0*\x,3*\y){\begin{tikzpicture}[scale=0.5]\node[sp] (0) at (0,0){};\node[sp] (1) at (1,0){};\node[sp] (2) at (2,0){};\node[sp] (3) at (3,0){};\draw[se] (0) -- (1);\draw[se] (1) -- (2);\draw[se] (2) -- (3);\end{tikzpicture}};
        
        \node (p12) at (-1*\x,2*\y){\begin{tikzpicture}[scale=0.5]\node[sp] (0) at (0,0){};\node[sp] (1) at (1,0){};\node[sp] (2) at (2,0){};\node[sp] (3) at (3,0){};\draw[se] (0) -- (1);\draw[se] (1) -- (2);\end{tikzpicture}};
        \node (p13) at (0*\x,2*\y){\begin{tikzpicture}[scale=0.5]\node[sp] (0) at (0,0){};\node[sp] (1) at (1,0){};\node[sp] (2) at (2,0){};\node[sp] (3) at (3,0){};\draw[se] (0) -- (1);\draw[se] (2) -- (3);\end{tikzpicture}};
        \node (p23) at (1*\x,2*\y){\begin{tikzpicture}[scale=0.5]\node[sp] (0) at (0,0){};\node[sp] (1) at (1,0){};\node[sp] (2) at (2,0){};\node[sp] (3) at (3,0){};\draw[se] (1) -- (2); \draw[se] (2) -- (3);\end{tikzpicture}};
        \node (p1) at (-1*\x,1*\y){\begin{tikzpicture}[scale=0.5]\node[sp] (0) at (0,0){};\node[sp] (1) at (1,0){};\node[sp] (2) at (2,0){};\node[sp] (3) at (3,0){};\draw[se] (0) -- (1);\end{tikzpicture}};
        \node (p2) at (0*\x,1*\y){\begin{tikzpicture}[scale=0.5]\node[sp] (0) at (0,0){};\node[sp] (1) at (1,0){};\node[sp] (2) at (2,0){};\node[sp] (3) at (3,0){};\draw[se] (1) -- (2);\end{tikzpicture}};
        \node (p3) at (1*\x,1*\y){\begin{tikzpicture}[scale=0.5]\node[sp] (0) at (0,0){};\node[sp] (1) at (1,0){};\node[sp] (2) at (2,0){};\node[sp] (3) at (3,0){};\draw[se] (2) -- (3);\end{tikzpicture}};
        \node (p0) at (0*\x,0*\y){\begin{tikzpicture}[scale=0.5]\node[sp] (0) at (0,0){};\node[sp] (1) at (1,0){};\node[sp] (2) at (2,0){};\node[sp] (3) at (3,0){};\end{tikzpicture}};
        \draw[thick, dotted] (p123) -- (p12) -- (p1) -- (p0);
        \draw[thick,dotted] (p123) -- (p23) -- (p2) -- (p0);
        \draw[thick,dotted] (p123) -- (p13) -- (p3) -- (p0);
        \draw[thick,dotted] (p12) -- (p2);
        \draw[thick,dotted] (p23) -- (p3);
        \draw[thick,dotted] (p13) -- (p1);
	\end{tikzpicture}

	\caption{The coherently oriented linear graph $\tI_3$ (top left), the multipath complex $X(\tI_3)$ (top right), and the path poset $P(\tI_3)$ (bottom).}
	\label{fig:nstep}
\end{figure}
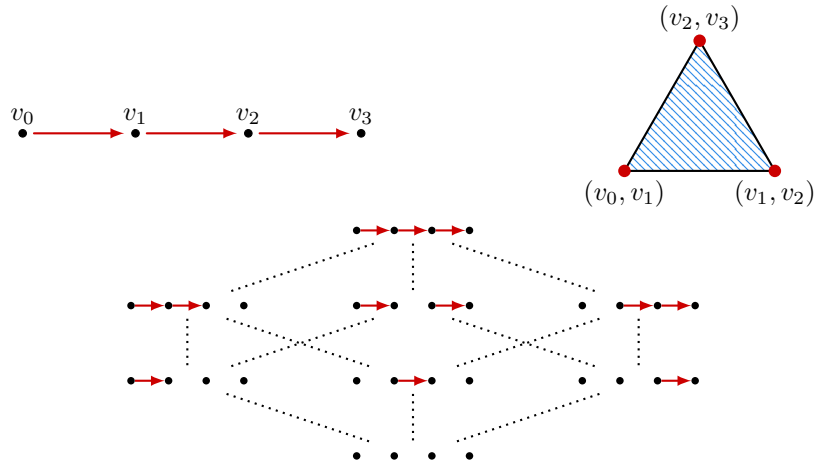

Equivalently, $X(\tG)$ is the simplicial complex on the edges of $\tG$ whose $k$-simplices are given by the multipaths in
$\tG$ of length $k-1$. When $\tG$ is the complete digraph, the associated simplicial complex is also called \emph{cycle free chessboard complex}, denoted $\Omega_n$ in~\cite{Omega}. 
Since being a multipath is a monotone property of digraphs, that is it is closed under subgraphs -- cf.~\cite{bjorner_cdg} or \cite{Jonsson}, it follows that $X(\tG)$ is a well-defined simplicial complex.

\begin{lemma}\label{lem:functX}
The association $\tG\mapsto X(\tG)$ yields a functor
\[
X\colon \Quiver^\op\to\mathbf{SimplComp}
\]
 to the  category of simplicial complexes and simplicial maps. 
\end{lemma}

\begin{proof}
Let $\phi\colon\tG'\to \tG$ be a minor morphism of quivers. Then, the morphism $\phi^\op$ is injective on the set of edges: for $e$ an edge of $\tG$, $\phi^\op(e)$ is again an edge, of $\tG'$. Furthermore, $\phi^\op$ preserves the adjacency relation, hence it sends simple paths to simple paths, and, in turn, multipaths to multipaths. Lengths and containment relations are also preserved, from which we get a simplicial map $X(\tG)\to X(\tG')$. As the construction preserves identities and compositions, the statement follows.
\end{proof}

Denote by $\ast$ the join operation of simplicial complexes -- cf.~\cite[Definition~2.16]{Kozlov}. Then, for quivers $\tG$ and $\tH$, we have  isomorphisms  
\begin{equation}\label{eq:joinmulti}
X(\tG\sqcup \tH)\cong X(\tG)\ast X(\tH) \ ,
\end{equation}
where $\sqcup$ denotes the disjoint union; see also  \cite[Remark~2.6]{jason}. 
We observe that the disjoint union on $\Quiver$ and the join on ${\bf SimplComp}$ yields  monoidal structures on the respective categories. 
Equaition~\eqref{eq:joinmulti} implies that the functor $X$ is symmetric monoidal, up to isomorphism of simplicial complexes. 

\begin{example}
Consider the coherently oriented linear digraph $\tI_n$ -- see Figure~\ref{fig:nstep} for an illustration of the quiver~$\tI_3$.
The path poset~$(P(\tI_n), \leq)$ is isomorphic to the Boolean poset~$\mathbb{B}(n)$. Thence, the associated multipath complex is an $(n-1)$-dimensional simplex.
Likewise, consider the coherently oriented polygonal digraph $\tP_n$ with $n$ edges, obtained from $\tI_{n}$ by identifying the vertices $v_0$ and $v_{n}$. Then, the path poset $(P(\tP_n),\leq )$ is isomorphic to the Boolean poset $\mathbb{B}(n)$ minus its maximum, and the corresponding multipath complex is a $(n-2)$-dimensional sphere.
Further, if we consider the minor morphism $\phi\colon \tI_n \to \tI_{n-1}$ given by the contraction of the edge $(v_{n-1},v_n)$, then $X(\phi^{\op})$ is the inclusion of the $(n-2)$-dimensional simplex into the boundary of the $(n-1)$-dimensional simplex as the subcomplex of~the latter spanned by all vertices, but one.
\end{example}

As often happens to simplicial complexes associated to monotone properties of graphs, it is easy to find multipath complexes which are homotopy equivalent to wedges of spheres, and, in particular, torsion free. Examples of such complexes are those associated to directed trees. In fact, as we shall see below, the multipath complex of a directed tree is homotopy equivalent to the join of matching complexes of undirected trees, cf.~Proposition~\ref{prop:multipathmatchingtree}. On the other hand,  matching complexes of trees are contractible or homotopy equivalent to wedges of spheres~\cite[Theorem 4.13]{MR2426164}. It follows that the multipath complex of a directed tree is either contractible or homotopy equivalent to a wedge of spheres. 
Further, the complexity of wedges of spheres which can be realised, up to homotopy, as multipath complexes of directed trees can be arbitrarily high; for instance, for each $n$, multipath complexes of directed trees homotopy equivalent to the wedge of an arbitrary number of $n$-dimensional spheres are given in the proof of~\cite[Proposition~6.12]{secondo}.

The fact that the multipath complex of an alternating digraph coincides with the matching complex of the underlying undirected graph~\cite{monotonecohm22} can be used to find multipath complexes whose homology is not torsion free:

\begin{example}
Let $\tK_{n,m}$ be the complete bipartite (undirected) graph on $m,n$ vertices. When equipped with an alternating orientation, the multipath complex~$X(\tK_{n,m})$ coincides with the matching complex of $\tK_{n,m}$ and its homology groups can have torsion. The minimal values of $m$ and $n$ for which the homology groups of $X(\tK_{n,m})$ have torsion are $m=n=5$, and in such case there is $3$-torsion~\cite{torsion}. 
\end{example}

The torsion in the homology of matching complexes of undirected graphs with genus less or equal than a fixed number was shown to be bounded in~\cite[Theorem~1.2]{miyata2023graph}. Therefore, when restricting to complete bipartite graphs with alternating orientation, the torsion of multipath complexes is also bounded. We want to extend this result to categories of quivers with bounded genus. In the spirit of the techniques presented in Section~\ref{sec:Grobner_digraph}, we first need to show that the homology of multipath complexes yields a finitely generated module.

\begin{prop}\label{prop:homologyfgmulti}
For each $i\in \bN$, the homology functor     \[
    \mathrm{H}_i(X(-);\bZ)\colon \Quiver^\op_{\leq g} \to \Mod_R
    \]
    is a finitely generated  $\Quiver^\op_{\leq g}$-module.
\end{prop}

\begin{proof}
    The proof follows the same ideas as \cite[Theorem~4.3]{miyata2023graph}. 
First, we observe that taking chain complexes and homology in a fixed degree is functorial with respect to taking minors. Using Lemma~\ref{lem:functX}, we infer that the functor $\mathrm{H}_i(X(-);\bZ)$ is a $\Quiver^\op_{\leq g}$-module. Furthermore,  $\mathrm{H}_i(X(-);\bZ)$  is a subquotient of the tensor product~$\mathcal{E}_{g}^{\otimes i}$ of the edge module; this module is finitely generated over $\Quiver^\op_{\leq g}$ by Proposition~\ref{prop:edgemodfg}. The statement follows by Theorem~\ref{rem:qGrobner is Noeth}.
\end{proof}

As a consequence of Proposition~\ref{prop:homologyfgmulti}, we have that torsion of multipath complexes is bounded:

\begin{prop}\label{prop:torsionmultipaths}
For every pair of integers $i, g\geq 0$, there exists $m = m(g,i) \in \bZ$ which annihilates the torsion subgroup of $H_i(X(\tG);\bZ)$, for each quiver~$\tG$ of genus at most $g$. 
\end{prop}

\begin{proof}
The proof follows the arguments of \cite[Theorem~3.20]{miyata2023graph}, which we include for the sake of completeness. 
Denote by $T_i(\tG)$ the torsion submodule of the homology of multipath complex of $\tG$ in (homological) degree $i$. The functor~$T_i$ is a $\Quiver_{\leq g}^{\op}$-submodule of~$\mathrm{H}_i(X(-);\bZ)$. Then, in virtue of Proposition~\ref{prop:homologyfgmulti} and by~Theorem~\ref{rem:qGrobner is Noeth}, it follows that~$T_i$ is a finitely generated $\Quiver_{\leq g}^{\op}$-module.
By Lemma~\ref{lem:fingengen}, there is a finite number of quivers $\tG_1,\dots,\tG_k$, such that there is a surjective group homomorphism
\[ \bigoplus_{j=1}^{k} T_i(\tG_j) \twoheadrightarrow T_i(\tG)\, ,\]
for each quiver $\tG$ of genus at most $g$. 
Thus, the least common multiple of the exponents of $T_i(\tG_1), \dots, T_i(\tG_k)$ gives the desired~$m(g,i)$. 
\end{proof}

We wish to point out  here that  the statements  in Proposition~\ref{prop:homologyfgmulti} and Proposition~\ref{prop:torsionmultipaths}   can be extended to  monotone properties~$P$ of directed graphs, or quivers, which satisfy the following additional property:  if $\phi\colon \tG\to \tG'$  is a minor
morphism and $\sigma'$ is a simplex in $X_P(\tG')$,
then the subgraph of $\tG$ induced by the image $\phi^*(E(\sigma'))$ is also in $P$. Here we have denoted by $X_P(\tG)$ the simplicial complex associated to $\tG$ via the monotone proeperty~$P$. In the case of the multipaths, note that the corresponding monotone property satisfies this condition. 
   Another interesting monotone property of digraphs is given by (the complex of) directed forests. In such case, as by~\cite{zbMATH05569076},  the complex of directed forests of graphs without oriented cycles is shellable; hence, it has no torsion. This observation leads to the following question, whose answer is not known to the authors:

    \begin{q}
        Does the complex of directed forests of graphs with oriented cycles contain torsion?
    \end{q}

The terminology of this last part of the section is taken from \cite{jason}, and we will limit our discussion to digraphs (rather than all quivers).

Let $\tG$ be a digraph, and let $\tG'\leq \tG$ be a subgraph.  The \emph{complement} $\Comp{\tG'}{\tG}$ of  $\tG'$ in $\tG$ is the subgraph of $\tG$ spanned by the edges in $E(\tG)\setminus E(\tG')$. The \emph{boundary} $\partial_{\tG} \tG'$ of $\tG'$ in $\tG$, or simply $\partial \tG'$ when $\tG$ is clear from the context, is defined as~$\partial_{\tG} \tG'=V(\tG')\cap V(\Comp{\tG'}{\tG})$. 

\begin{defn}
Let $\tG$ be a connected digraph with at least one edge and without self-loops. A vertex $v\in V(\tG)$ is called \emph{stable} if it is not, at the same time, a source and a target (i.e.~if it does not belong to the intersection $s(E(\tG))\cap t(E(\tG))$), and is \emph{unstable} otherwise. 
A \emph{dynamical region} in $\tG$ is a connected subgraph $\tR \leq \tG$, with at least one edge, such that: 
\begin{enumerate}[label = (\alph*)]
    \item\label{item:dandelion} all vertices in the boundary of $\tR$ are unstable in $\tG$, but stable in both $\tR$ and $\Comp{\tR}{\tG}$;
    \item\label{item:cicle} no edge of $\tR$ belongs to any oriented cycle in $\tG$ which is not contained in $\tR$.
\end{enumerate}
A dynamical region is called \emph{stable} if all its non-boundary vertices are stable. It   is called \emph{unstable} if all its non-boundary vertices are~unstable, and it has at least one non-boundary vertex.
A dynamical region which has no proper subgraph which is itself a dynamical region is called a \emph{dynamical module}.
\end{defn}

\begin{lemma}\label{lemma:stabledynmod}
Let $\tM$ be a dynamical module of a connected directed tree. Then, $\tM$ is stable.
\end{lemma}

\begin{proof}
Let $\tT$ be a directed tree, and $\tM\leq \tT$ a dynamical module. 
Assume, by contradiction, that $\tM$ is not stable. Then, there is at least an unstable vertex $v$ not in the boundary of $\tM$. Since $v$ is an unstable vertex in $\tM$ there are some edges, say $e_1, \dots ,e_h \in E(\tM)$, such that~$s(e_i)=v$, and some edges, say $f_1, \dots, f_k \in E(\tM)$, such that~$t(f_j)=v$, for $k,h\geq 1$. 

Since $\tT$ is a directed tree, so is $\tM$. In particular, deleting the edges $e_1,...,e_h$ disconnects $\tM$. Denote by $\tR$ the connected component containing the edges $f_1,...,f_k$.
We argue that $\tR$ is a dynamical region of $\tM$, and this yields the desired contradiction.
First, $\tR$ being a subgraph of $\tM \leq \tT$ implies that item (b) in the definition of dynamical region is automatically satisfied. 
Observe that $v$ is the unique vertex in the boundary of $\tR$ in $\tM$.
In fact, given $w\neq v$ in $V(\tR)$ all edges of $\tM$ incident in $w$ are also edges of $\tR$. Otherwise, one of the $e_i$'s would be incident to $w$ and we would have a closed loop in (the geometric realisation of) $\tT$, e.g.~by taking an (unoriented) path from $v$ to $w$ in $\tR$ and then adding $e_i$. 
Since $v$ is unstable in $\tM$ and is, by construction, stable both in $\tR$ and in $C_{\tM}(\tR)$, the statement follows.  
\end{proof}

Stable dynamical regions are alternating digraphs.
Therefore, the multipath complex of a stable dynamical region $\tR$ in a directed graph $\tG$ is the matching complex of the underlying
undirected graph of $\tR$ -- cf.~\cite[Lemma~4.8]{jason}. As a consequence of this fact, we get the following:

\begin{prop}\label{prop:multipathmatchingtree}
    Let $\tT$ be a connected directed tree. Then, $X(\tT)$ is homotopy equivalent to a join of matching complexes of (undirected) trees.
\end{prop}

\begin{proof}
    By \cite[Theorem~1]{jason}, there is a unique (up to re-ordering) decomposition of $\tG$ into dynamical modules $\tM_1,...,\tM_{k}$, and we have an isomorphism
 \[X(\tG) \cong X(\tM_1) \ast \cdots \ast X(\tM_k) \ .\]
 By Lemma~\ref{lemma:stabledynmod}, each $\tM_i$ is stable, hence each $X(\tM_i)$ is the matching complex of the underlying undirected tree of $\tM_i$. The statement follows.
\end{proof}

A straightforward corollary follows, strengthening some of the results contained in \cite{jason}.

\begin{cor}\label{cor:forestscontrspheres}
    The multipath complex of a forest is either contractible or homotopy equivalent to a wedge of spheres. 
 \end{cor}

Hence, a direct consequence of Proposition~\ref{prop:multipathmatchingtree} is that the homology of multipath complexes of directed trees does not contain any torsion. In order to get interesting torsion, one needs to consider non-acyclic quivers.  However, when the quivers contain directed cycles, computations of the homology groups are more involved -- and often computationally exponential. Therefore, it would be interesting to know more about the global  growth of the ranks of the homology groups. By applying similar arguments as in \cite[Proposition~4.3]{proudfoot2022contraction}, the ranks have a polynomial growth, and they could be packed into a Hilbert series.
Let us recall the definition of Hilbert series. 
Denote by ${\bf C}$ a(n essentially) small category. A~\emph{norm} on ${\bf C}$ is a function $\nu$ from the objects of ${\bf C}$ modulo isomorphism to $\bN$. 

\begin{defn} 
Given a finitely generated ${\bf C}$-module (over $R$) $\mathcal{F}$ and a norm $\nu$ on ${\bf C}$, the \emph{Hilbert series} (with respect to $\mathcal{F}$ and $\nu$) is defined as
\[ {S}_{\mathcal{F},\nu}(t) = \sum_{x\in {\bf C}} \text{rank}_R(\mathcal{F}(x))t^{\nu(x)}\ . \]
If ${\bf C}$ is a category of graphs we will take $\nu$ to be the number of edges, and omit it from the notation. 
\end{defn}

{We know from \cite{Hilbert_series} that the Hilbert series of matching complexes of undirected trees is algebraic. 
By Proposition~\ref{prop:multipathmatchingtree}, multipath complexes of directed trees decompose, up to isomorphism of simplicial complexes, as join products of matching complexes of trees. As a consequence, we get that their homology over a field can be expressed in terms of the K\"unneth formula. This might suggest that, in complete analogy to what happens with matching complexes, also the Hilbert series of multipath complexes is algebraic.  }
\begin{q}
    {Denote by ${\bf T}_k$ the subcategory of $\CDigraph$ spanned by trees with at most $k$ unstable vertices. Is the Hilbert series with respect to $\H_i(X_{-};\bZ)\colon {\bf T}_k \to R\text{-}{\bf Mod}$ algebraic for each choice of $i,k\in \bN$?}
\end{q}

\subsection{Multipath complexes and matching complexes}
In the previous section we used the observation, from~\cite{monotonecohm22}, that the multipath complexes of trees are strictly related to matching complexes. The aim of this section is to extend the relation between multipath complexes and matching complexes to a wider class of quivers, and, as a consequence, to generalise  Proposition~\ref{prop:multipathmatchingtree}. We will make use of a blow-up operation. To avoid further technical assumptions,  we  work also in this section with digraphs rather than with quivers.

Recall that the \emph{indegree} (resp. \emph{outdegree}) of a vertex $v$ is the number of edges with target (resp.~source)~$v$.

\begin{defn}\label{def:blowup}
Let $\tG$ be a digraph, and let $v\in V(\tG)$ with both indegree and outdegree different from $0$. The \emph{blow-up} of $\tG$ at $v$ is the digraph $B(\tG,v)$ obtained from $\tG$ as follows: the vertices of $B(\tG,v)$ are the same as $\tG$ except for $v$, which is replaced by $v_{\mathrm{in}}$ and $v_{\mathrm{out}}$. Given $v',v''\in V(B(\tG,v))$ the edges from $v'$ to $v''$ are in bijection with 
\begin{enumerate}[label = (\alph*)]
\item the edges between the corresponding vertices in $\tG$, if $\{v',v''\}\cap \{v_{\mathrm{in}},v_{\mathrm{out}}\} = \emptyset$;
\item the edges from $v'$ to $v$, if $v'\neq v_{\mathrm{out}}$ and $v''=v_{\mathrm{in}}$;
\item the edges from $v$ to $v''$, if $v'=v_{\mathrm{out}}$ and $v''\neq v_{\mathrm{in}}$;
\item the empty set, in the remaining cases.
\end{enumerate}
If the indegree or the outdegree of $v$ are zero, then we set $B(\tG,v) \coloneqq \tG$.
The \emph{blow-up} of $\tG$ is the digraph $B(\tG)$ obtained from $\tG$ by iteratively blowing-up all vertices in $\tG$ one after the other; see Figure~\ref{fig:blowup} for an illustrative example.
\end{defn}

Note that the blow-up of digraphs does not depend on the order we blow-up the vertices. Furthermore, blowing-up a vertex $v$ yield a pair of vertices $v_{\mathrm{in}}$ and $v_{\mathrm{out}}$. The blow-up of either $v_{\mathrm{in}}$ or $v_{\mathrm{out}}$ yields just the identity.
Therefore, the blow-up is a trivial operation on the class of alternating digraphs.

\begin{rem}\label{rem:blowup_is_alternating}
    Blow-ups of digraphs are alternating digraphs, and, in particular, bipartite.
\end{rem}

Note that blow-ups of digraphs might yield disconnected digraphs. In fact, it is not difficult to check that the connected components of the blow-up $B(\tG)$ of a digraph $\tG$ without oriented cycles are in one-to-one correspondence with the dynamical modules of $\tG$.

\begin{example}
    The blow-up of $\tI_n$ is the disconnected digraph $\bigsqcup_{i=1}^n\tI_1$. While the blow-up of $\tA_n$ is $\tA_n$.
\end{example}

Locally, blowing-up a vertex corresponds to blowing-up a dandelion subdigraph at its central vertex:

\begin{example}
   Consider the dandelion 
$\tD_{n,m}$ on $(n+m+1)$ vertices,  and $(m+n)$ edges, defined as follows:
\begin{enumerate}
    \item $V(\tD_{n,m}) = \{ v_{0}, w_{1} ,...., w_{n}, x_{1}, ..., x_{m} \}$;
    \item $E(\tD_{n,m}) = \{ (w_i,v_0), (v_{0},x_j) \mid i=1,...,n; j=1,...,m \}$.
\end{enumerate}
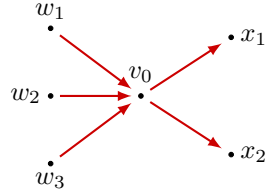
\begin{figure}[h]
	\begin{tikzpicture}[scale=0.6][baseline=(current bounding box.center)]
		\tikzstyle{point}=[circle,thick,draw=black,fill=black,inner sep=0pt,minimum width=2pt,minimum height=2pt]
		\tikzstyle{arc}=[shorten >= 8pt,shorten <= 8pt,->, thick]
		
		\node (v0) at (0,0) {};
		\node[above] at (0,0.1) {$v_0$};
		\draw[fill] (v0)  circle (.05);
		\node (w1) at (-2,1.5) {};
		\node[above] at (-2,1.5) {$w_1$};
		\draw[fill] (-2,1.5)  circle (.05);
		\node  (w2) at (-2,0) { };
		\node[left]  at (-2,0) {$w_{2}$};
		\draw[fill] (-2,0)  circle (.05);
		\node  (w3) at (-2,-1.5) { };
		\node[below]  at (-2,-1.5) {$w_{3}$};
		\draw[fill] (-2,-1.5)  circle (.05);
		
		\node  (x1) at (2,1.3) { };
		\node[right]  at (2,1.3) {$x_1$};
		\draw[fill] (2,1.3)  circle (.05);
		
		\node  (x2) at (2,-1.3) { };
		\node[right]  at (2,-1.3) {$x_2$};
		\draw[fill] (2,-1.3)  circle (.05);

		\draw[thick, bunired, -latex] (w1) -- (v0);
		\draw[thick, bunired, -latex] (w2) -- (v0);
		\draw[thick, bunired, -latex] (w3) -- (v0);
		\draw[thick, bunired, -latex] (v0) -- (x1);
		\draw[thick, bunired, -latex] (v0) -- (x2);
	\end{tikzpicture}
	\caption{The  graph $\tD_{3,2}$.}
	\label{fig:nmgraph}
\end{figure}
In other words we have a single $(m+n)$-valent vertex $v_0$, all remaining vertices are univalent, there are $n$ edges with target $v_0$, and there are $m$ edges with source $v_0$ -- cf.~Figure~\ref{fig:nmgraph}. Then, the blow-up of $\tD_{n,m}$ is the blow-up~$B(\tD_{n,m},v_0)$ of  $\tD_{n,m}$ at the central vertex $v_0$, which is the disjoint union of a source with $m$ edges, and a sink with $n$ edges. Note that maximal multipaths of $\tD_{n,m}$ are of type $\{(w_i,v_0),(v_0,x_j)\}$, whereas  maximal  multipaths of $B(\tD_{n,m},v_0)$ are of type $\{\{(w_i,v_0)\},\{(v_0,x_j)\}\}$. 
\end{example}

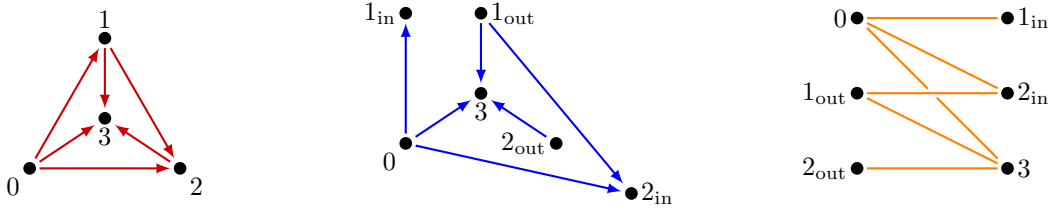
\begin{figure}[h]
    \centering
    \begin{tikzpicture}[scale = 2, thick]
    \node (a) at (0,0) {};
    \node (b) at (.5,.866) {};
    \node (c) at (1,0) {};
    \node (d) at (.5,.333) {};
    
    \node at (0,0) [below left] {$ 0$};
    \node at (.5,.866) [above] {$ 1$};
    \node at (1,0) [below right] {$ 2$};
    \node at (.5,.333) [below] {$3$};

    \draw[black, fill] (a) circle (.035);
    \draw[black, fill] (b) circle (.035);
    \draw[black, fill] (c) circle (.035);
    \draw[black, fill] (d) circle (.035);

    \draw[-latex, bunired] (a) -- (b);
    \draw[-latex, bunired] (a) -- (c);
    \draw[-latex, bunired] (b) -- (c);
    \draw[-latex, bunired] (a) -- (d);
    \draw[-latex, bunired] (b) -- (d);
    \draw[-latex, bunired] (c) -- (d);
    
    \begin{scope}[shift = {+(2.5,0.166)}]

    \node (a) at (0,0) {};
    \node (bout) at (.5,.866) {};
    \node (bin) at (0,.866) {};
    \node (cout) at (1,0) {};
    \node (cin) at (1.5,-.333) {};
    \node (d) at (.5,.333) {};
    
    \node at (0,0) [below left] {$ 0$};
    \node at (.5,.866) [right] {$1_{\mathrm{out}}$};
    \node at (1,0) [left] {$2_{\mathrm{out}}$};
    \node at (0,.866) [left] {$1_{\mathrm{in}}$};
    \node at (1.5,-.333) [ right] {$2_{\mathrm{in}}$};
    \node at (.5,.333) [below] {$3$};

    \draw[black, fill] (a) circle (.035);
    \draw[black, fill] (bin) circle (.035);
    \draw[black, fill] (cin) circle (.035);
    \draw[black, fill] (bout) circle (.035);
    \draw[black, fill] (cout) circle (.035);
    \draw[black, fill] (d) circle (.035);

    \draw[-latex, blue] (a) -- (bin);
    \draw[-latex, blue] (a) -- (cin);
    \draw[-latex, blue] (bout) -- (cin);
    \draw[-latex, blue] (a) -- (d);
    \draw[-latex, blue] (bout) -- (d);
    \draw[-latex, blue] (cout) -- (d);
    \end{scope}

        \begin{scope}[shift = {+(5.5,0)}]

    \node (a) at (0,1) {};
    \node (b) at (0,0.5) {};
    \node (c) at (0,0) {};
    \node (ap) at (1,1) {};
    \node (bp) at (1,0.5) {};
    \node (cp) at (1,0) {};
    
    \node at (0,1) [left] {$ 0$};
    \node at (0,0.5) [left] {$1_{\mathrm{out}}$};
    \node at (0,0) [left] {$2_{\mathrm{out}}$};
    \node at (1,1) [right] {$1_{\mathrm{in}}$};
    \node at (1,0.5) [right] {$2_{\mathrm{in}}$};
    \node at (1,0) [right] {$3$};

    \draw[black, fill] (a) circle (.035);
    \draw[black, fill] (b) circle (.035);
    \draw[black, fill] (c) circle (.035);
    \draw[black, fill] (ap) circle (.035);
    \draw[black, fill] (bp) circle (.035);
    \draw[black, fill] (cp) circle (.035);

    \draw[orange] (a) -- (ap);
    \draw[orange] (a) -- (bp);
    \draw[orange] (a) -- (cp);
        \draw[white, fill] (0.5,0.5) circle (.035);
    \draw[orange] (b) -- (bp);
    \draw[orange] (b) -- (cp);
    \draw[orange] (c) -- (cp);
    \end{scope}
    
\end{tikzpicture}
\caption{Transitive tournament $\tT_3$ (in red on the left), its blow-up (in blue at the center), and the associated bipartite graph $\tB_3$ (in orange on the right).}
\label{fig:blowup}
\end{figure}

\begin{prop}\label{prop:singleblup}
Let $\tG$ be a digraph. 
Suppose that $v\in V(\tG)$ does not belong to any oriented cycle in $\tG$. Then, the multipath complex $X(\tG)$ is canonically isomorphic, as a simplicial complex, to $X(B(\tG,v))$.
\end{prop}
\begin{proof}
By definition, there is a natural bijection $b\colon E(\tG) \to E(B(\tG,v))$ given by
\[ b((v_1,v_2)) = \begin{cases}
(v_1, v_{\mathrm{in}}) & \text{if }v_2 = v ,\\
(v_{\mathrm{out}}, v_{2}) & \text{if }v_1 = v ,\\
(v_1,v_2)  & \text{otherwise}
\end{cases}\]
for each $(v_1,v_2)\in E(\tG)$.
Given a multipath $\tM$ in $\tG$, we define $b(\tM)$ as the spanning subgraph of $B(\tG,v)$ spanned by the edges in $b(E(\tM))\subseteq E(B(\tG,v))$.
To prove the statement is sufficient to show that $b$ gives a well-defined bijection between multipaths in $\tG$ and multipaths in $B(\tG,v)$.

The image of a multipath in $\tG$ is a multipath in $B(\tG,v)$;
if at most one edge in the multipath has $v$ as endpoint, then the statement is clear. 
If two edges of the multipath have $v$ as endpoint, then the connected component of the multipath containing these two edges splits into the disjoint union of two paths (one containing $v_{\mathrm{in}}$ and the other containing $v_{\mathrm{out}}$). 

We claim that every multipath $\tM'$ in $B(\tG,v)$ is the image of a multipath in $\tG$. 
To this end we show that there is at most one edge in each connected component of $\tM'$ that has one endpoint in $\{ v_{\mathrm{in}}, v_{\mathrm{out}} \}$. 
Assume, by contradiction, that there are two edges of the multipath, say $e_1$ and $e_2$, incident in $v_{\mathrm{out}}$ and $v_{\mathrm{in}}$, respectively. 
If $e_1$ and $e_2$ belong to the same component $\gamma$ of $\tM'$ we get a contradiction. 
By construction, there is no edge $e$ in $B(\tG,v)$ such that $t(e) = v_{\mathrm{out}}$ or $s(e) = v_{\mathrm{in}}$.
Thus, $e_1$ is the first edge in $\gamma$ and $e_2$ the last edge in $\gamma$.
Therefore, $b^{-1}(\gamma)$ is an oriented cycle in $\tG$ which has $v$ as a vertex. This is absurd since $v$ does not belong to any oriented cycle in $\tG$. 
It follows that $b^{-1}(\tM')$ is a well-defined multipath, and the claim follows. 

Since $b$ is a bijection between the edges of the two digraphs, the induced map on multipaths is injective. This concludes the proof.
\end{proof}

We are now ready to state the main technical result of the section. Recall by Remark~\ref{rem:adjunctiongraphs} that the functor $\iota$ associates to a quiver the underlying undirected graph. 

\begin{thm}\label{thm:multi=match}
Let $\tG$ be a digraph without oriented cycles. Then, the multipath complex~$X(\tG)$ is isomorphic, as simplicial complex, to the matching complex of $\iota(B(\tG))$. 
\end{thm}

\begin{proof}
Since no vertex of $\tG$ is contained in any oriented cycle, by Proposition~\ref{prop:singleblup}, $X(\tG)$ is canonically isomorphic to~$X(B(\tG))$. 
The digraph $B(\tG)$, and, in particular, each connected component of $B(\tG)$, is alternating -- cf.~Remark~\ref{rem:blowup_is_alternating}. Hence, by \cite[Theorem~4.1]{monotonecohm22}, for each connected component $\tH$ of $B(\tG)$, we have an isomorphism
of simplicial complexes
$M (\iota(\tH)) \cong X(\tH)$. Let $\tH_1,\dots,\tH_n$ be the connected components of $B(\tG)$. Then, by Equation~\eqref{eq:joinmulti}, we have
\[
X(\tG)\equiv  X(\tH_1)\ast\dots \ast X(\tH_n)\cong M(\iota(\tH_1))\ast\dots \ast M(\iota(\tH_n))\ ,
\]
where $M$ denotes the matching complex. 
As also the matching complex of a disjoint union is, up to isomorphism, the join product of the matching complexes, this fact implies the statement.
\end{proof}

Let $\Digraph_o$ be the category of directed graph without oriented cycles, and \emph{regular} morphisms of directed graphs; that means, injective maps $\phi\colon \tG_1\to\tG_2$ of digraphs  such that  $(v,w) \in E(\tG_1) \implies \phi(v)\neq \phi(w)$. {Note, that this is the same as the opposite category of digraphs without cycles, and deletions.}

\begin{rem}\label{rem:functblowup}
    The blow-up construction lifts to a functor 
    \[
    \mathbf{B} \coloneqq B\circ \iota \colon \Digraph_o\to \Graph^\op \ .
    \]
    In fact, it is easy to see that the blow-up  extends to regular morphims of digraphs, and preserves compositions and identities. Note that we do not allow contractions here because contractions may generate new oriented cycles. 
\end{rem}

\begin{cor}\label{cor:commdiagram}
    The following diagram
    \begin{center}
	\begin{tikzcd}
		\Digraph_o\arrow[r,"\mathbf{B}"]\arrow[dr,"X"']  & \Graph^\op\arrow[d,"M"]\\
		&  \mathbf{SimpCompl}
	\end{tikzcd}
 \end{center}
    is commutative up to isomorphism of simplicial complexes.
\end{cor}

\begin{proof}
The statement follows in view of Theorem~\ref{thm:multi=match} and Remark~\ref{rem:functblowup}.
\end{proof}

We can use the identification with the matching complex given in Theorem~\ref{thm:multi=match} to get new  computations of the homotopy type of matching complexes. In fact, let $\tB_n$ be the undirected graph on $2n$ vertices $p_1, \dots, p_n$ and $q_1, \dots, q_n$, and edges $(p_i, q_j)$ for all $i\leq j$. Sometimes, these are called  half-graphs or ladders~\cite{NESETRIL2021103223}.

\begin{thm}\label{thm:matchingBn}
    The matching complex of $\tB_n$ is either contractible or homotopy equivalent to a wedge of spheres. 
\end{thm}
\begin{proof}
     Let $\tT_n$ be the transitive tournament on $n+1$ vertices  (with only edges of type $(i,j)$ for $i<j$). Then, we can identify  its blow-up with the bipartite  graph $\tB_n$ -- see also Figure~\ref{fig:blowup}. In light of Theorem~\ref{thm:multi=match}, the multipath complex of $\tT_n$ is isomorphic to the matching complex of $\iota(\tB_n)$. The statement now follows directly from~\cite[Theorem~5.1]{jason}. 
\end{proof}

\begin{rem}
    In view of Theorem~\ref{thm:matchingBn} and \cite[Lemma~5.7]{jason}, we get that the matching complexes of bipartite graphs obtained from $\tB_n$ by deletion of a subset of vertices in $\{ p_1,\dots,p_n\}$ are either contractible or homotopy equivalent to wedges of spheres. 
\end{rem}

 Let $\tG$ be a digraph without oriented cycles. Assume that, for each embedded cycle (possibly of length~$2$), the induced orientation on it is not alternating. In such case, we say that $\tG$ is a \emph{digraph without oriented and alternating cycles}. In particular, $\tG$ does not have pairs of edges of the form $(v,w)$ and $(w,v)$.
With this terminology, we directly get the following: 

\begin{thm}\label{thm:Hilbertarb}
    Let $\tG$ be a digraph  without any oriented or alternating cycle. 
    Then, the multipath complex of $\tG$ is either contractible or a wedge of spheres.
\end{thm}

\begin{proof}
    By assumption, the induced orientation on each embedded cycle is not alternating. This means that, for each embedded cycle~$\tH$, we can find a vertex~$v_\tH$ with both indegree and outdegree non zero. Blowing-up $\tG$ at $v_\tH$ breaks the cycle $\tH$.
    Therefore, the blow-up of $\tG$ is a directed forest. The statement follows by  Theorem~\ref{thm:multi=match} and \cite[Theorem~4.13]{MR2426164}.
\end{proof}

The class of digraphs without oriented and alternating graphs is not big. For example, cones of cycles or of trees with at least an unstable vertex contain oriented or alternating cycles. 
Therefore, we can not apply Theorem~\ref{thm:Hilbertarb} to such digraphs, and cone graphs may contain torsion. However, in complete analogy with Proposition~\ref{prop:torsionmultipaths}, we get that such torsion is also bounded:

\begin{prop}
For every pair of integers $k, g\geq 0$, there exists $m = m(g,k) \in \bZ$ which annihilates the torsion subgroup of $\H_i(X(\Cone(\tQ));\bZ)$, for each quiver~$\tQ$  of genus at most $g$.     
\end{prop}

\begin{proof}
    Use Proposition~\ref{prop:ConeEdge} and then proceed as for \cite[Theorem~1.2]{miyata2023graph}.
\end{proof}

\section{Applications to magnitude homology}\label{sec:magnitude}

We start by recalling the definition of magnitude homology, following~\cite{richardHHA, asao}. 

A directed graph $\tG$ can be seen as a metric space with the path metric~$d$, where $d(v,w)$ is the minimum length across all directed paths between $v$ and $w$. If $v$ and $w$ are not connected by a directed path, we set $d(v,w)\coloneqq \infty$. 
Given a non-negative integer $k$ and a commutative ring $R$, let $\Lambda_k(\tG;R)\coloneqq R\langle (v_0,\dots, v_k)\mid v_i\in V(\tG)\rangle$ be the $R$-module freely generated by $(k+1)$-tuples of vertices of $\tG$. We admit the empty tuple, and set $\Lambda_k(\tG;R)\coloneqq 0$ for $k\leq -2$.
For a $k$-tuple $(v_0,\dots,v_k)$ of vertices of $\tG$, with $v_i\neq v_{i+1}$ and $d(v_i,v_{i+1})<\infty$ for each $i$, the \emph{length} of $(v_0,\dots,v_k)$ in $\tG$ is the number 
\[
\ell(v_0,\dots,v_k)\coloneqq \sum_{i=0}^{k-1}d(v_i,v_{i+1}) \ .
\]
Furthermore, we  define a differential on $\Lambda_k(G;R)$ by setting
    \[
    \delta(v_0,\dots,v_k)\coloneqq \sum_{i=1}^{k-1} (-1)^i\delta_i(v_0,\dots,v_k) \ ,
    \]
    where $\delta_i(v_0,\dots,v_k)=(v_0,\dots,v_{i-1},v_{i+1}, \dots, v_k)$  if $\ell(v_0,\dots,v_k)=l=\ell(v_0,\dots,v_{i-1},v_{i+1}, \dots, v_k)$, and it is set to $0$ otherwise.

  Consider the submodule 
$
I_k(\tG;R)\coloneqq R\langle (v_0,\dots, v_k)\mid v_i=v_{i+1} \text{ for some } i \rangle
$
of $\Lambda_k(\tG;R)$, where $I_0(\tG)$ is set to~$ 0$.   
The family of such  modules can be equipped with the differential $\delta$, yielding a chain complex. As we have the inclusions $I_k(\tG;R)\subseteq \Lambda_k(\tG;R)$ for all $k$, we can form the quotient chain complex with modules $R_k(\tG;R)\coloneqq \Lambda_k(\tG;R)/ I_k(\tG;R)$. The magnitude chain complex $\MC_{k,l}(\tG;R)$ is defined as the submodule of $R_k(\tG;R)$ given by all those tuples of length~$l$;  this is compatible with the chain complex structure -- {cf.}~\cite[Lemma~2.14]{asao}.
Hence, the pair $(\MC_{*,l}(\tG;R), \delta)$ is a chain complex, and its homology called \emph{magnitude homology} (of~$\tG$) with coefficients in $R$.

\begin{rem}
    A contraction $\phi\colon \tG\to \tH$ of directed graphs induces a chain map
\[
\phi_\# \colon \MC_{*,*}(\tG;R)\to \MC_{*,*}(\tH;R) 
\]
which, to the tuple $(v_0,\dots,v_k)$ of $\tG$, associates the tuple $(\phi(v_0),\dots,\phi(v_k))$ if $\ell(\phi(v_0),\dots,\phi(v_k))= \ell(v_0,\dots,v_k)$, and it is set to be $0$ otherwise. The map $\phi_\#$ is a chain map, as it commutes with the differential $\delta$, and it induces a map
in magnitude homology.
\end{rem}

 For the following result, we refer the reader to \cite[Proposition~3.3]{richardHHA} or \cite[Lemma~3.6]{asao}.

\begin{prop}
    Magnitude homology is a functor
    \[
    \MH_{*,*}\colon \CDigraph\to \mathbf{BiGrMod}_R
    \]
    from the contraction category of digraphs to the category of bigraded $R$-modules.
\end{prop}

The dual version of magnitude homology, called \emph{magnitude cohomology}
\[ {\rm MC}_{l}^{k}(\tG; R) \coloneqq {\rm Hom}({\rm MC}_{l,k}(\tG; R); R)\ ,\]
was introduced in~\cite{MagnitudeCohomology}. It also 
yields a functor 
  $
    \MH_{*}^{*}\colon \CDigraph^{\op} \to \mathbf{BiGrMod}_R 
    $.

    We have recalled here the definition of magnitude homology of directed graphs. Almost \emph{verbatim} we could have defined magnitude homology of undirected graphs. In fact, the two notions are related:  
    
\begin{rem}\label{rem:magnitudegraphsquivers}
    Let $\tG$ be an (undirected) graph. Then, in the notation of Remark~\ref{rem:adjunctiongraphs}, we have that the magnitude (co)homology of the digraph~$\rho(\tG)$ coincides with the magnitude (co)homology (of undirected graphs) of $\tG$ -- see also \cite[Remark~2.16]{asao}. 
\end{rem}

More generally,  the definitions of magnitude homology and magnitude cohomology can also be  extended to quivers, yielding functors on the whole category $\CQuiver$ of quivers and contractions.   Let $R$ be a  commutative Noetherian ring, with identity. Then, we have the following extension of \cite[Theorem~3.11]{torsionCC}:

\begin{thm}\label{thm:mgfg}
The $\CQuiver^{\op}_{\leq g}$-module  ${\rm MH}^{k}_{l}(-;R)\colon \CQuiver^{\op}_{\leq g}\to \Mod_R$ is finitely generated. 
\end{thm} 

\begin{proof}
  By Theorem~\ref{thm:contractcatfg}, the (opposite) category $\CQuiver^{\op}_{\leq g}$ of quivers of bounded genus, and contractions, is quasi-Gr\"obner. Hence, subquotients of finitely generated $\CQuiver^{\op}_{\leq g}$-modules are finitely generated. By Theorem~\ref{thm:Vkfg}, the $\CQuiver^{\op}_{\leq g}$-module $\Hom(\mathcal{V}^{\oplus k},R)$  is finitely generated. The statement now follows using the same arguments of   \cite[Proposition~3.10]{torsionCC}, adapted to the case of quivers.
 \end{proof}

Using Theorem~\ref{thm:mgfg}, we can extend the results of \cite{torsionCC} to the more general setting of quivers. In particular, a straightforward adaptation of the arguments in \cite{torsionCC} yields:

\begin{cor}\label{cor:torsionmagnitude}
For every pair of integers $k, g\geq 0$, there exists $m = m(g,k) \in \bZ$ which annihilates the torsion subgroup of $\MH^{k}_{*}(\tG;\bZ)$, for each quiver~$\tG$ of genus at most $g$. 
\end{cor}

\begin{cor}\label{cor:jafbka}
    Let $\bK$ be a field, and $g\geq 0$. Then, there exists a polynomial $f\in \bZ[t]$ of degree at most~$g+1$, such that, for all quivers $\tG$ of genus at most $ g$, we have
    \[
    \dim_\bK \MH_*^k(\tG;\bK) \leq   f(\# E(\tG)) \ ,
\]
where $\# E(\tG)$ is the number of edges of $\tG$.
\end{cor}

Corollary~\ref{cor:torsionmagnitude} says that says that the order of torsion classes in integral magnitude (co)homology of quivers of genus at most $g$, in a fixed (co)homological
degree $k$, is bounded. By Remark~\ref{rem:magnitudegraphsquivers} and using the results of \cite{MR4275098} and, in particular, \cite[Theorem~3.13]{SadzanovicSummers}, we know that magnitude homology and cohomology of quivers can contain torsion. However, we do not know if such torsion always comes from undirected graphs via the functor $\rho$, or not. 

\begin{q}
Is it possible to extend the constructions of \cite{MR4275098} and \cite{SadzanovicSummers} so to get torsion of directed graphs which does not come from the essential image of $\rho\colon \Graph\to \Quiver$ of Remark~\ref{rem:adjunctiongraphs}?
\end{q}

In view of the algebraicity results of Section~\ref{sec:multipath}, we ask the following:

\begin{q}
      Is the Hilbert series of the magnitude homology of (directed) graphs  algebraic?
\end{q}

\begin{q}
    The categories $\Graph$ and $\Quiver$ are related by the functors $\iota$ and $\rho$. What is the relation between generators of magnitude homology of graphs and magnitude homology of quivers?
\end{q}

Magnitude homology has strong connections with another homological invariant of digraphs, the so-called \emph{path homology} -- {cf.}~\cite{Grigoryan_first}. More precisely, path homology is the diagonal component of a certain spectral sequence involving magnitude homology~\cite{asao}. Despite torsion in magnitude homology of digraphs is yet unclear, when we turn to path homology we have more information; to be more precise, we have the following: 

\begin{rem}
Path homology of digraphs can contain arbitrary torsion: if $X$ is a simplicial complex, and $\tG_X$ the Hasse digraph associated to $X$, then the path homology of $\tG_X$ is isomorphic to the (simplicial) homology of $X$ -- see, {e.g.}, \cite{hha/1401800084}, or \cite{Grigoryan2016} for the cohomological version. 
\end{rem}

The following corollary tells us that also in the case of path homology of digraphs, torsion has to be bounded:

\begin{cor}\label{cor:torionPH}
For each $g,k$ positive integers, there exists a $d = d({g,k}) \in \bZ$ such that, for each digraph $\tG$ of genus~$g$, the torsion part of the path cohomology $\mathrm{PH}^k(\tG,\bZ)$  has exponent at most $d$.
\end{cor}

\begin{proof}
    Path (co)homology groups appear  as groups in the second page of a spectral sequence whose $0$-th page features magnitude chain groups. Turning the pages of the spectral sequence corresponds to taking  subsequent subquotients of the $0$-th page. Therefore, the statement follows thanks to Theorem~\ref{thm:mgfg}.
\end{proof}

\bibliographystyle{alpha}
\bibliography{biblio}
\end{document}